\documentclass[final]{birkjour}
\usepackage{amsfonts}
\usepackage{amsthm}
\usepackage{amssymb}
\usepackage{amsmath}
\usepackage{upref}
\usepackage{mathrsfs}
\usepackage{color}
\usepackage[citecolor=blue,colorlinks=true]{hyperref}
\usepackage{enumitem}

\usepackage[notref,notcite]{showkeys}

\newtheorem{thm}{Theorem}[section]
\newtheorem{lemma}[thm]{Lemma}
\newtheorem{prop}[thm]{Proposition}
\newtheorem{cor}[thm]{Corollary}
\theoremstyle{definition}
\newtheorem{defin}[thm]{Definition}
\theoremstyle{remark}
\newtheorem{rem}[thm]{Remark}

\numberwithin{equation}{section}

\newcommand{\dif}{\mathrm{d}}
\newcommand{\totdif}{\mathrm{D}}
\newcommand{\mf}{\mathscr{F}}
\newcommand{\me}{\mathrm{e}}
\newcommand{\mr}{\mathbb{R}}
\newcommand{\prst}{\mathbb{P}}
\newcommand{\stred}{\mathbb{E}}
\newcommand{\ind}{\mathbf{1}}
\newcommand{\mn}{\mathbb{N}}

\newcommand{\mt}{\mathbb{T}}

\DeclareMathOperator{\supp}{supp}
\DeclareMathOperator*{\esssup}{ess\,sup}

\DeclareMathOperator{\diver}{div}

\newcommand{\semicol}{;}

\renewcommand{\theenumi}{(\roman{enumi})}
\renewcommand{\labelenumi}{(\roman{enumi})}

\begin{document}
\title[A BGK Approximation to Stochastic Scalar Conservation Laws]{A Bhatnagar-Gross-Krook Approximation\\to Stochastic Scalar Conservation Laws}

\author{Martina Hofmanov\'a}

\address{Department of Mathematical Analysis\\ Faculty of Mathematics and Physics, Char\-les University\\ Sokolovsk\'a~83\\ 186~75 Praha~8\\ Czech Republic\vspace{2mm}}
\address{ENS Cachan Bretagne, IRMAR, CNRS, UEB\\ av. Robert Schuman\\ 35~170 Bruz\\ France\vspace{2mm}}
\address{Institute of Information Theory and Automation of the ASCR\\ Pod~Vod\'arenskou v\v{e}\v{z}\'i~4\\ 182~08 Praha~8\\ Czech Republic}

\email{martina.hofmanova@bretagne.ens-cachan.fr}

\thanks{This research was supported in part by the GA\,\v{C}R Grant no. P201/10/0752 and the GA UK Grant no. 556712.}

\begin{abstract}
We study a BGK-like approximation to hyperbolic conservation laws forced by a multiplicative noise. First, we make use of the stochastic characteristics method and establish the existence of a solution for any fixed parameter $\varepsilon$. In the next step, we investigate the limit as $\varepsilon$ tends to $0$ and show the convergence to the kinetic solution of the limit problem.
\end{abstract}

\maketitle

\section{Introduction}

In the present paper, we consider a scalar conservation law with stochastic forcing
\begin{equation}\label{conser}
\begin{split}
\dif u+\diver\big(A(u)\big)\dif t&=\varPhi(u)\,\dif W,\qquad t\in(0,T),\,x\in\mt^N,\\
u(0)&=u_0
\end{split}
\end{equation}
and study its approximation in the sense of Bhatnagar-Gross-Krook (a BGK-like approximation for short). In particular, we aim to describe the conservation law \eqref{conser} as the hydrodynamic limit of the stochastic BGK model, as the microscopic scale $\varepsilon$ goes to $0$.

% Here, the BGK model has became an important tool in developing kinetic numerical schemes to solve hyperbolic conservation laws. 

The literature devoted to the deterministic counterpart, i.e. corresponding to the situation $\varPhi=0$, is quite extensive (see \cite{vov1}, \cite{vov2}, \cite{lpt1}, \cite{lions}, \cite{nouri}, \cite{nouri1}, \cite{tadmor}, \cite{perth}). In that case, the BGK model is given as follows
\begin{equation}\label{bgkdet}
\big(\partial_t+a(\xi)\cdot\nabla\big)f^\varepsilon=\frac{\chi_{u^\varepsilon}-f^\varepsilon}{\varepsilon},\qquad t>0,\,x\in\mt^N,\,\xi\in\mr,
\end{equation}
where $\chi_{u^\varepsilon}$, the so-called equilibrium function, is defined by
$$\chi_{u^\varepsilon}(\xi)=\ind_{0<\xi<u^\varepsilon}-\ind_{u^\varepsilon<\xi<0},$$
and $a$ is the derivative of $A$. The differential operator $\nabla$ is with respect to the space variable $x$. The additional real-valued variable $\xi$ is called velocity; the solution $f^\varepsilon$ is then a microscopic density of particles at $(t,x)$ with velocity $\xi$. The local density of particles is defined by
\begin{equation*}
u^\varepsilon(t,x)=\int_\mr f^\varepsilon(t,x,\xi)\,\dif \xi.
\end{equation*}
The collisions of particles are given by the nonlinear kernel on the right hand side of \eqref{bgkdet}.
The idea is that, as $\varepsilon\rightarrow 0$, the solutions $f^\varepsilon$ of \eqref{bgkdet} converge to $\chi_u$ where $u$ is the unique kinetic or entropy solution of the deterministic scalar conservation law.

The addition of the stochastic term to the basic governing equation is rather natural for both practical and theoretical applications. Such a term can be used for instance to account for numerical and empirical uncertainties and therefore stochastic conservation laws has been recently of growing interest, see \cite{bauzet}, \cite{debus}, \cite{feng}, \cite{holden}, \cite{kim}, \cite{stoica}, \cite{wittbolt}, \cite{weinan}.
%The first existence and uniqueness results concerned entropy solutions in the case of one space dimension and additive noise \cite{kim}, later also multiplicative noise \cite{feng} and the case of multi-dimensional bounded domain \cite{wittbolt}.
The first complete well-posedness result for multi-dimensional scalar conservation laws driven by a general multiplicative noise was obtained by Debussche and Vovelle \cite{debus} for the case of kinetic solutions. In the present paper, we extend this result and show that the kinetic solution is the macroscopic limit of stochastic BGK approximations. As the latter are much simpler equations that can be solved explicitly, this analysis can be used for developing innovative numerical schemes for hyperbolic conservation laws. 
%The first existence result concerning one space dimension was given by Holden and Risebro \cite{holden}. Kim \cite{kim} then showed existence of a unique entropy solution in the case of one space dimension and one-dimensional additive noise. Feng and Nualart \cite{feng} studied the case of multiplicative noise and proved existence of an entropy solution in dimension one and uniqueness in any dimension. Vallet and Wittbolt \cite{wittbolt} showed existence and uniqueness in the case of one-dimensional additive noise and a bounded domain in any space dimension. Debussche and Vovelle \cite{debus} gave a complete well-posedness result for kinetic solutions in any space dimension and multiplicative noise.

The BGK model in the stochastic case reads
\begin{equation}\label{bgk}
\begin{split}
\dif F^\varepsilon+a(\xi)\cdot\nabla F^\varepsilon\,\dif t&=\frac{\ind_{u^\varepsilon>\xi}-F^\varepsilon}{\varepsilon}\,\dif t-\partial_\xi F^\varepsilon\varPhi\,\dif W-\frac{1}{2}\partial_\xi\big(G^2(-\partial_\xi F^\varepsilon)\big)\,\dif t,\\
F^\varepsilon(0)&=F^\varepsilon_0,
\end{split}
\end{equation}
where the function $F^\varepsilon$ corresponds to $f^\varepsilon+\ind_{0>\xi}$, the local density $u^\varepsilon$ is given as above, and the function $G^2$ will be defined in \eqref{linrust}. Note, that setting $\varPhi=0$ in \eqref{bgk} yields an equation which is equivalent to the deterministic BGK model \eqref{bgkdet}.
Our purpose here is twofold. First, we make use of the stochastic characteristics method as developed by Kunita in \cite{kun1} to study a certain auxiliary problem. With this in hand, we fix $\varepsilon$ and prove the existence of a unique weak solution to the stochastic BGK model \eqref{bgk}.
Second, we establish a series of estimates uniform in $\varepsilon$ which together with the results of Debussche and Vovelle \cite{debus} justify the limit argument, as $\varepsilon\rightarrow0$, and give the convergence of the weak solutions of \eqref{bgk} to the kinetic solution of \eqref{conser}.

Let us make some comments on the deterministic BGK model \eqref{bgkdet}. Even though the general concept of the proof is analogous, we point out that the techniques required by the stochastic case are significantly different. In particular, the characteristic system for the deterministic BGK model consists of independent equations
$$\frac{\dif x_i(t)}{\dif t}=a_i(\xi),\qquad i=1,\dots,\,N,$$
and the $\xi$-coordinate of the characteristic curve is constant. Accordingly, it is much easier to control the behavior of $f^\varepsilon$ for large $\xi$. Namely, if the initial data $f_0^\varepsilon$ are compactly supported (in $\xi$), the same remains valid also for the solution itself and also the convergence proof simplifies. On the contrary, in the stochastic case, the $\xi$-coordinate of the characteristic curve is governed by an SDE and therefore this property is, in general, lost. Similar issues has to be dealt with in order to obtain all the necessary uniform estimates. To overcome this difficulty, it was needed to develop a suitable method to control the decay at infinity in connection with the remaining variables $\omega,\,t,\,x$. (cf. Proposition \ref{kin}).

Using this approach we are able to prove the convergence of the BGK model under a slightly weaker hypothesis on the initial datum $u_0$ than usually assumed in the deterministic case: it is not supposed to be bounded, we only assume $u_0\in L^p(\Omega\times\mt^N)$ for all $p\in[1,\infty)$. Note, that under this condition, the initial data for the deterministic BGK model, for instance $f_0^\varepsilon=\chi_{u_0}$, are not compactly supported and so the usual methods are not applicable. In the deterministic case, however, the boundedness assumption is fairly natural since also the solution $u$ to the conservation law remains bounded. Obviously, this is not true for the stochastic case as it is impossible to get any $L^\infty_\omega$ estimates due to the active white noise term.

There is another difficulty coming from the complex structure of the characteristic system for the stochastic BGK model \eqref{bgk}. Namely, the finite speed of propagation that is an easy consequence of boundedness of the solution $u$ of the conservation law in the deterministic case (see for instance \cite{tadmor}) is no longer valid and therefore some growth assumptions on the transport coefficient $a$ are in place. The hypothesis of bounded derivatives is natural for the stochastic characteristics method as it implies the existence of global stochastic flows. Even though this already includes one important example of Burgers' equation it is of essential interest to handle also more general coefficients having polynomial growth. This was achieved by a suitable cut-off procedure which also guarantees all the necessary estimates.

The exposition is organized as follows. In Section \ref{setting}, we introduce the basic setting and state the main result, Theorem \ref{main}. In order to make the paper more self-contained, Section \ref{prelim} provides a brief overview of two concepts which are the keystones of our proof of existence and convergence of the BGK model. On the one hand, it is the notion of kinetic solution to stochastic hyperbolic conservation laws, on the other hand, the method of stochastic characteristics for first-order linear SPDEs. Section \ref{existence} is mainly devoted to the existence proof for stochastic BGK model, however, in the Subsection \ref{further} we establish some important estimates useful in Section \ref{convergence}. This final section contains technical details of the passage to the limit and completes the proof of Theorem \ref{main}.

\section{Setting and the main result}
\label{setting}
%
%To begin with, we introduce some notation. We denote by $C^{l,\delta}$ the space of all $l$-times continuously differentiable functions with $\delta$-H\"{o}lder continuous $l$-th derivatives. By $C_b^{l,\delta}$ we denote its subspace containing functions with bounded derivatives up to order $l$ (the function itself is only required to be of linear growth).

We now give the precise assumptions on each of the terms appearing in the above equations \eqref{conser} and \eqref{bgk}. We work on a finite-time interval $[0,T],$ $T>0,$ and consider periodic boundary conditions: $x\in\mt^N$ where $\mt^N$ is the $N$-dimensional torus.
The flux function
$$A=(A_1,\dots,A_N):\mr\longrightarrow\mr^N$$
is supposed to be of class $C^{4,\eta}$, for some $\eta>0$, with a polynomial growth of its first derivative, denoted by $a=(a_1,\dots,a_N)$.

Regarding the stochastic term, let $(\Omega,\mf,(\mf_t)_{t\geq0},\prst)$ be a stochastic basis with a complete, right-continuous filtration. The initial datum may be random in general, i.e. $\mf_0$-measurable, and we assume $u_0\in L^p(\Omega;L^p(\mt^N))$ for all $p\in[1,\infty)$. As we intend to apply the stochastic characteristics method developed by Kunita \cite{kun1}, we restrict ourselves to finite-dimensional noise. However, our results extend to infinite-dimensional setting once the corresponding properties of stochastic flows are established.
Let $\mathfrak{U}$ be a finite-dimensional Hilbert space and $(e_k)_{k=1}^d$ its orthonormal basis. The process $W$ is a $d$-dimensional $(\mf_t)$-Wiener process: $W(t)=\sum_{k=1}^d \beta_k(t)\, e_k$ with $(\beta_k)_{k=1}^d$ being mutually independent real-valued standard Wiener processes relative to $(\mf_t)_{t\geq 0}$. The diffusion coefficient $\varPhi$ is then defined as 
\begin{equation*}
\begin{split}
\varPhi(z):\mathfrak{U}&\longrightarrow L^2(\mt^N)\\
h&\longmapsto \sum_{k=1}^d g_k(\cdot,z(\cdot))\langle e_k,h\rangle,\qquad z\in L^2(\mt^N),
\end{split}
\end{equation*}
where the functions $g_1,\dots,g_d:\mt^N\times\mr\rightarrow\mr$ are of class $C^{4,\eta}$, for some $\eta>0$, with linear growth and bounded derivatives of all orders. Under these assumptions, the following estimate holds true
\begin{equation}\label{linrust}
G^2(x,\xi)=\sum_{k=1}^d|g_k(x,\xi)|^2\leq C\big(1+|\xi|^2\big),\qquad x\in\mt^N,\,\xi\in\mr.
\end{equation}
However, in order to get all the necessary estimates (cf. Corollary \ref{indik}, Remark \ref{indikreason}), we restrict ourselves to two special cases: either
\begin{equation}\label{nula}
g_k(x,0)=0,\qquad x\in\mt^N,\,k=1,\dots,d,
\end{equation}
hence \eqref{linrust} rewrites as
$$G^2(x,\xi)\leq C|\xi|^2,\qquad x\in\mt^N,\,\xi\in\mr,$$
or we strengthen \eqref{linrust} in the following way
\begin{equation}\label{omez}
G^2(x,\xi)\leq C,\qquad x\in\mt^N,\,\xi\in\mr.
\end{equation}
Note, that the latter is satisfied for instance in the case of additive noise.

In this setting, we can assume without loss of generality that the $\sigma$-algebra $\mf$ is countably generated and $(\mf_t)_{t\geq 0}$ is the completed filtration generated by the Wiener process and the initial condition. Let us denote by $\mathcal{P}$ the predictable $\sigma$-algebra on $\Omega\times[0,T]$ associated to $(\mf_t)_{t\geq0}$
%is defined as follows
%$$\mf_t=\sigma\big(\sigma(u_0)\cup\sigma\big(W_{r_1}-W_{r_2};\,0\leq r_1,r_2\leq t\big)\big),\qquad t\geq 0.$$
%Let $\mathcal{P}$ denote the predictable $\sigma$-algebra on $\Omega\times[0,T]$ associated to $(\mf_t)_{t\geq0}$. 
%Next, we denote by $(\mathscr{G}_{s,t})_{t\geq s\geq0}$ the completed two parametric filtration generated by $W$
%% for any $s\in[0,T]$ we denote by $(\mathscr{G}_{s,t})_{t\geq0}$ the filtration generated by the Wiener process $W(t)-W(s)$, $t\geq s$:
%% by $(\mathscr{G}_{s,t})_{t\geq s\geq 0}$ the two parametric filtration generated by $W$
%%$$\mathscr{G}_{s,t}=\sigma\big(W_{r_1}-W_{r_2}\semicol\, s\leq r_1,r_2\leq t\big),\qquad 0\leq s\leq t<\infty,$$
%%and set
%$$\mathscr{G}_{s,t}=\sigma\big(\sigma \big(W(r_1)-W(r_2);\,s\leq r_1,r_2\leq t\big)\cup\{N\in\mf;\,\prst(N)=0\}\big)$$
%and set
%$$\mf_{s,t}=\sigma\big(\mf_{s}\cup\mathscr{G}_{s,t}\big),\qquad 0\leq s\leq t<\infty,$$
%i.e. $(\mf_{s,t})_{t\geq s}$ is the filtration to be used for solving an equation with the initial time $s$ (cf. Corollary \ref{weaksol}, Corollary \ref{weaksol1}). Indeed, it accommodates any $\mf_{s}$-measurable initial condition as well as the corresponding Wiener process that starts from $s$: $W(t)-W(s)$, $t\geq s$.
and by $\mathcal{P}_s$ the predictable $\sigma$-algebra on $\Omega\times[s,T]$ associated to $(\mf_{t})_{t\geq s}$.
For notational simplicity, we write $L^\infty_{\mathcal{P}_s}(\Omega\times[s,T]\times\mt^N\times\mr)$ to denote\footnote{$\mathcal{B}(\mt^N)$ and $\mathcal{B}(\mr)$, respectively, denotes the Borel $\sigma$-algebra on $\mt^N$ and $\mr$, respectively.}
$$L^\infty\big(\Omega\times[s,T]\times\mt^N\times\mr,\mathcal{P}_s\otimes\mathcal{B}(\mt^N)\otimes\mathcal{B}(\mr),\dif\prst\otimes\dif t\otimes\dif x\otimes\dif \xi\big).$$
%and similarly $L^\infty_{\mathcal{P}}(\Omega\times[0,T]\times\mt^N\times\mr)$ for
%$$L^\infty\big(\Omega\times[0,T]\times\mt^N\times\mr,\mathcal{P}\otimes\mathcal{B}(\mt^N)\otimes\mathcal{B}(\mr),\dif\prst\otimes\dif t\otimes\dif x\otimes\dif \xi\big).$$
%where $\mathcal{P}$ was defined above.

Concerning the initial data for the BGK model \eqref{bgk}, one possibility is to consider simply $F^\varepsilon_0=\ind_{u_0>\xi}$, however, one can also take some suitable approximations of $\ind_{u_0>\xi}$. Namely, let $\{u_0^\varepsilon;\,\varepsilon\in(0,1)\}$ be a set of approximate $\mf_0$-measurable initial data, which is bounded in $L^p(\Omega;L^p(\mt^N))$ for all $p\in[1,\infty)$, and assume in addition that $u_0^\varepsilon\rightarrow u_0$ in $L^1(\Omega;L^1(\mt^N))$. Thus, setting $F_0^\varepsilon=\ind_{u^\varepsilon_0>\xi}$, $f_0^\varepsilon=\chi_{u_0^\varepsilon}$ yields the convergence $f_0^\varepsilon\rightarrow f_0=\chi_{u_0}$ in $L^1(\Omega\times\mt^N\times\mr)$.

Let us close this section by stating the main result to be proved precisely.

\begin{thm}[Hydrodynamic limit of the stochastic BGK model]\label{main}
Let the above assumptions hold true. Then, for any $\varepsilon>0$, there exists $F^\varepsilon\in L^\infty_\mathcal{P}(\Omega\times[0,T]\times\mt^N\times\mr)$ which is a unique weak solution to the stochastic BGK model \eqref{bgk} with initial condition $F_0^\varepsilon=\ind_{u^\varepsilon_0>\xi}$. Furthermore, if $f^\varepsilon=F^\varepsilon-\ind_{0>\xi}$ then $(f^\varepsilon)$ converges in $L^p(\Omega\times[0,T]\times\mt^N\times\mr)$, for all $p\in[1,\infty)$, to the equilibrium function $\chi_u$, where $u$ is the unique kinetic solution to the stochastic hyperbolic conservation law \eqref{conser}. Besides, the local densities $(u^\varepsilon)$ converge to the kinetic solution $u$ in $L^p(\Omega\times[0,T]\times\mt^N)$, for all $p\in[1,\infty)$.
\end{thm}

Throughout the paper, we use the letter $C$ to denote a generic positive constant,
which can depend on different quantities but $\varepsilon$ and may change from one line to another.
We also employ a shortened notation for various $L^p$-type norms, e.g. we write $\|\cdot\|_{L^p_{\omega,x,\xi}}$ for the norm in $L^p(\Omega\times\mt^N\times\mr)$ and similarly for other spaces.

\section{Preliminary results}
\label{prelim}

As we are going to apply the well-posedness theory for kinetic solutions of hyperbolic scalar conservation laws \eqref{conser} as well as the theory of stochastic flows generated by stochastic differential equations, we provide a brief overview of these two concepts.

%For the reader's convenience, we provide a brief overview of two important concepts used later on in the proof of the main result. Namely, the notion of kinetic solution and kinetic formulation for scalar conservation laws on the one hand and stochastic flows  and stochastic characteristics method on the other.

\subsection{Kinetic formulation for scalar conservation laws}

The main reference for this subsection is the paper of Debussche and Vovelle \cite{debus}. For further reading about the kinetic approach used in different settings, we refer the reader to \cite{chen}, \cite{hof}, \cite{lpt1}, \cite{lions}, or \cite{perth}.
In the paper \cite{debus}, the notion of kinetic and generalized kinetic solution to \eqref{conser} was introduced and the existence, uniqueness and continuous dependence on initial data were proved. In the following, we present the main ideas and results while skipping all the technicalities.

Let $u$ be a smooth solution to \eqref{conser}. It follows from the It\^o formula that $u$ also satisfies the kinetic formulation of \eqref{conser}
\begin{equation}\label{kinetic}
\partial_t F+a(\xi)\cdotp\nabla F=\delta_{u=\xi}\varPhi(u)\dot{W}+\partial_\xi\bigg(m-\frac{1}{2}G^2\delta_{u=\xi}\bigg),
\end{equation}
where $F=\ind_{u>\xi}$ and $m$ is an unknown kinetic measure, i.e. a random nonnegative bounded Borel measure on $[0,T]\times\mt^N\times\mr$ that vanishes for large $\xi$ in the following sense: if $B_R^c=\{\xi\in\mr;\,|\xi|\geq R\}$ then
\begin{equation*}
\lim_{R\rightarrow\infty}\stred\,m\big(\mt^N\times[0,T]\times B_R^c\big)=0.
\end{equation*}
Hence we arrive at the notion of kinetic solution: let $u\in L^p(\Omega\times[0,T],\mathcal{P},\dif\prst\otimes\dif t;L^p(\mt^N))$, $\forall p\in[1,\infty)$. It is said to be a kinetic solution to \eqref{conser} provided $F=\ind_{u>\xi}$ is a solution, in the sense of distributions over $[0,T]\times\mt^N\times\mr$, to the kinetic formulation \eqref{kinetic} for some kinetic measure $m$. Replacing the indicator function by a general kinetic function $F$ we obtain the definition of a generalized kinetic solution. It corresponds to the situation where one does not know the exact value of $u(t,x)$ but only its law given by a probability measure $\nu_{t,x}$. More precisely, let $F(t),\,t\in[0,T],$ be a kinetic function on $\Omega\times\mt^N\times\mr$ and $\nu_{t,x}(\xi)=-\partial_\xi F(t,x,\xi)$. Then $F$ is a generalized kinetic solution to \eqref{conser} provided: $F(0)=\ind_{u_0>\xi}$ and for any test function $\varphi\in C_c^\infty([0,T)\times\mt^N\times\mr)$,
\begin{equation}\label{general}
\begin{split}
\int_0^T&\big\langle F(t),\partial_t\varphi(t)\big\rangle\,\dif t+\big\langle F(0),\varphi(0)\big\rangle+\int_0^T\big\langle F(t),a(\xi)\cdotp\nabla\varphi(t)\big\rangle\,\dif t\\
&\qquad\qquad=-\sum_{k=1}^d\int_0^T\int_{\mt^N}\int_\mr g_k(x,\xi)\varphi(t,x,\xi)\,\dif\nu_{t,x}(\xi)\,\dif x\,\dif\beta_k(t)\\
&\quad-\frac{1}{2}\int_0^T\int_{\mt^N}\int_\mr G^2(x,\xi)\partial_\xi\varphi(t,x,\xi)\dif\nu_{t,x}(\xi)\,\dif x\,\dif t+m(\partial_\xi\varphi)
\end{split} 
\end{equation}
holds true $\prst$-a.s..
The assumptions considered in \cite{debus} are the following: the flux function $A$ is of class $C^1$ with a polynomial growth of its derivative$\semicol$ the process $W$ is a (generally infinite-dimensional) cylindrical Wiener process, i.e. $W(t)=\sum_{k\geq1}\beta_k(t) e_k$ with $(\beta_k)_{k\geq1}$ being mutually independent real-valued standard Wiener processes and $(e_k)_{k\geq1}$ a complete orthonormal system in a separable Hilbert space $\mathfrak{U}\semicol$ the mapping $\,\varPhi(z):\mathfrak{U}\rightarrow L^2(\mt^N)$ is defined for each $z\in L^2(\mt^N)$ by $\varPhi(z)e_k=g_k(\cdot,z(\cdot))$ where $g_k\in C(\mt^N\times\mr)$ and the following conditions
\begin{equation*}
\sum_{k\geq1}|g_k(x,\xi)|^2\leq C\big(1+|\xi|^2\big), 
\end{equation*}
\begin{equation*}\label{skorolip}
\sum_{k\geq1}|g_k(x,\xi)-g_k(y,\zeta)|^2\leq C\big(|x-y|^2+|\xi-\zeta|h(|\xi-\zeta|)\big),
\end{equation*}
are fulfilled for every $x,y\in\mt^N,\,\xi,\zeta\in\mr$, with $h$ being a continuous nondecreasing function on $\mr_+$ satisfying, for some $\alpha>0$,
\begin{equation*}\label{fceh}
h(\delta)\leq C\delta^\alpha,\quad\delta<1.
\end{equation*}
Under these hypotheses, the well-posedness result \cite[Theorem 11, Theorem 19]{debus} states: For any $u_0\in L^p(\Omega\times\mt^N)$ for all $p\in[1,\infty)$ there exists a unique kinetic solution to \eqref{conser}. Besides, any generalized kinetic solution $F$ is actually a kinetic solution, i.e. there exists a process $u$ such that $F=\ind_{u>\xi}$. Moreover, if $u_1,\,u_2$ are kinetic solutions with initial data $u_{1,0}$ and $u_{2,0}$, respectively, then for all $t\in[0,T]$
$$\stred\|u_1(t)-u_2(t)\|_{L^1_x}\leq\stred\|u_{1,0}-u_{2,0}\|_{L^1_x}.$$

\subsection{Stochastic flows and stochastic characteristics method}

\label{flows}

The results mentioned in this subsection are due to Kunita and can be found in \cite{kunita} and \cite{kun1}.
To begin with, we introduce some notation. We denote by $C_b^{l,\delta}(\mr^d)$ the space of all $l$-times continuously differentiable functions with bounded derivatives up to order $l$ (the function itself is only required to be of linear growth) and $\delta$-H\"{o}lder continuous $l$-th derivatives.

%
%We denote by $C^{l,\delta}(\mr^d)$ the space of all $l$-times continuously differentiable functions with $\delta$-H\"{o}lder continuous $l$-th derivatives. By $C_b^{l,\delta}(\mr^d)$ we denote its subspace containing functions with bounded derivatives up to order $l$ (the function itself is only required to be of a linear growth).

%\begin{defin}
%Let $\phi_{s,t}(y,\omega),\,s,\,t\in[0,T],\,y\in\mr^d,$ be a continuous $\mr^d$-valued random field. It is called a stochastic flow of homeomorphisms if there exists a null set $N$ of $\Omega$ such that for any $\omega\in N^c$, the family of continuous maps $\{\phi_{s,t}(\omega);\,s,\,t\in[0,T]\}$ defines a flow of homeomorphisms, i.e. it satisfies the following properties
%\begin{enumerate}
%\item $\phi_{s,t}(\omega)=\phi_{r,t}(\omega)\circ\phi_{s,r}(\omega)$ for all $r,\,s,\,t\in[0,T]$,
%\item $\phi_{s,s}(\omega)=\mathrm{Id}$ for all $s\in[0,T]$,
%\item $\phi_{s,t}(\omega):\mr^d\rightarrow\mr^d$ is an onto homeomorphism for all $s,\,t\in[0,T]$.
%\end{enumerate}
%Further, if $\phi_{s,t}(\omega)$ satisfies (iv), it is called a stochastic flow of $C^k$-diffeo\-morphisms
%\begin{enumerate}
%\item[(iv)] $\phi_{s,t}(\omega)$ is $k$-times differentiable with respect to $y$, for all $s,\,t\in[0,T]$, and the derivatives are continuous in $(s,t,x)$.
%\end{enumerate}
%\end{defin}

Let $B_t=(B^1_t,\dots,\,B^m_t)$ be an $m$-dimensional Wiener process and let $b^k:\mr^d\rightarrow\mr^d,$ $k=0,\dots,\,m.$ We study the following system of Stratonovich's stochastic differential equations
\begin{equation}\label{sde}
\dif \phi_t=b^0(\phi_t)\,\dif t+\sum_{k=1}^m b^k(\phi_t)\circ\dif B^k_t.
\end{equation}
Under the hypothesis that $b^1,\dots,\,b^m\in C_b^{l+1,\delta}(\mr^d)$ and $b^0\in C_b^{l,\delta}(\mr^d)$ for some $l\geq1$ and $\delta>0$, and for any given $y\in\mr^d$, $s\in[0,T]$, the problem \eqref{sde} possesses a unique solution starting from $y$ at time $s$.
Let us denote this solution by $\phi_{s,t}(y)$. It enjoys several important properties. Namely, it is a continuous $C^{l,\varepsilon}$-semimartingale for any $\varepsilon<\delta$ and defines a forward Brownian stochastic flow of $C^l$-diffeomorphisms, i.e. there exists a null set $N$ of $\Omega$ such that for any $\omega\in N^c$, the family of continuous maps $\{\phi_{s,t}(\omega);\,0\leq s\leq t\leq T\}$ satisfies
\begin{enumerate}
\item $\phi_{s,t}(\omega)=\phi_{r,t}(\omega)\circ\phi_{s,r}(\omega)$ for all $0\leq s\leq r\leq t\leq T$,
\item $\phi_{s,s}(\omega)=\mathrm{Id}$ for all $0\leq s\leq T$,
\item $\phi_{s,t}(\omega):\mr^d\rightarrow\mr^d$ is $l$-times differentiable with respect to $y$, for all $0\leq s\leq t\leq T$, and the derivatives are continuous in $(s,t,y)$,
\item $\phi_{s,t}(\omega):\mr^d\rightarrow\mr^d$ is a $C^l$-diffeomorphism for all $0\leq s\leq t\leq T$,
\item $\phi_{t_i,t_{i+1}}$, $i=0,\dots,\,n-1$, are independent random variables for any $0\leq t_0\leq\cdots\leq t_n\leq T$.
\end{enumerate}
Therefore, for each $0\leq s\leq t\leq T$, the mapping $\phi_{s,t}(\omega)$ has the inverse $\rho_{s,t}(\omega)=\phi_{s,t}(\omega)^{-1}$ which satisfies
\begin{enumerate}[resume]
\item \label{itm:vik} $\rho_{s,t}(\omega):\mr^d\rightarrow\mr^d$ is $l$-times differentiable with respect to $y$, for all $0\leq s\leq t\leq T$, and the derivatives are continuous in $(s,t,y)$,
\item $\rho_{s,t}(\omega)=\rho_{s,r}(\omega)\circ\rho_{r,t}(\omega)$ for all $0\leq s\leq r\leq t\leq T$,
\end{enumerate}
and consequently $\rho_{s,t}$ is a stochastic flow of $C^l$-diffeomorphisms for the backward direction. Indeed, the following holds true:
For any $0\leq s\leq t\leq T$, the process $\rho_{s,t}(y)$ satisfies the backward Stratonovich stochastic differential equation with the terminal condition $y$
\begin{equation*}
\rho_{s,t}(y)=y-\int_s^t b^0\big(\rho_{r,t}(y)\big)\,\dif r-\sum_{k=1}^m\int_s^t b^k\big(\rho_{r,t}(y)\big)\circ \hat{\dif\,}\!B^k_r,
\end{equation*}
where the last term is a backward Stratonovich integral defined by Kunita \cite{kun1} using the time-reversing method. To be more precise, the Brownian motion $B$ is regarded as a backward martingale with respect to its natural two parametric filtration
$$\sigma\big(B_{r_1}-B_{r_2}\semicol\,s\leq r_1,r_2 \leq t\big),\qquad 0\leq s\leq t\leq T,$$
the integral is then defined similarly to the forward case and both stochastic flows $\phi_{s,t}$ as well as $\rho_{s,t}$ are adapted to this filtration. Furthermore, we have a growth control for both forward and backward stochastic flow. Fix arbitrary $\delta\in(0,1)$, then the following convergences hold uniformly in $s,\,t,$ $\prst$-a.s.,
\begin{equation*}\label{vb1}
\lim_{|y|\rightarrow \infty}\frac{|\phi_{s,t}(y)|}{(1+|y|)^{1+\delta}}=0,\quad\qquad\lim_{|y|\rightarrow \infty}\frac{|\rho_{s,t}(y)|}{(1+|y|)^{1+\delta}}=0,
\end{equation*}
\begin{equation*}\label{vb2}
\lim_{|y|\rightarrow \infty}\frac{(1+|y|)^{\delta}}{1+|\phi_{s,t}(y)|}=0,\quad\qquad\lim_{|y|\rightarrow \infty}\frac{(1+|y|)^{\delta}}{1+|\rho_{s,t}(y)|}=0.
\end{equation*}

In the remainder of this subsection we will discuss the stochastic characteristics method where the theory of stochastic flows plays an important role. We restrict our attention to a first-order linear stochastic partial differential equation of the form
\begin{equation}\label{ch}
\begin{split}
\dif v&= b^{0}(y)\cdot\nabla_y v\,\dif t+\sum_{k=1}^m b^{k}(y)\cdot\nabla_y v\circ\dif B^k_t,\\
v(0)&=v_0,
\end{split}
\end{equation}
with coefficients $b^k:\mr^d\rightarrow\mr^d$, $k=0,\dots,\,m.$ The associa\-ted stochastic characteristic system is defined by a system of Stratonovich stochastic differential equations
\begin{equation}\label{sss}
\begin{split}
\dif \phi_t&=b^0(\phi_t)\,\dif t+\sum_{k=1}^m b^k(\phi_t)\circ\dif B^k_t,\\
%\dif \theta_t&=d^0(\phi_t)\,\theta_t\,\dif t+\sum_{k=1}^m d^k(\phi_t)\,\theta_t\circ\dif B^k_t.
\end{split}
\end{equation}
A solution of \eqref{sss} starting at $y$ is the so-called stochastic charac\-teristic curve of \eqref{ch} and will be denoted by $\phi_t(y)$. Assume that $b^1,\dots,\,b^m\in C_b^{l+1,\delta}(\mr^d)$ and $b^0\in C_b^{l,\delta}(\mr^d)$ for some $l\geq 3$ and $\delta>0$. If the initial function $v_0$ lies in $C^{l,\delta}(\mr^d)$, then the problem \eqref{ch} has a unique strong solution which is a continuous $C^{l,\varepsilon}$-semimartingale for some $\varepsilon>0$ and is represented by
\begin{equation}\label{explic}
v(t,y)=v_0\big(\phi_t^{-1}(y)\big),\qquad t\in[0,T],
\end{equation}
where the inverse mapping $\phi_t^{-1}$ is well defined according to the previous paragraph. It satisfies \eqref{ch} in the following sense
\begin{equation*}
\begin{split}
v(t,y)=v_0(y)+b^0(y)\cdot\int_0^t \nabla_y v(r,y)\,\dif r+\sum_{k=1}^m b^k(y)\cdot\int_0^t \nabla_y v(r,y)\circ\dif B^k_r.
\end{split}
\end{equation*}
Moreover, if the initial condition $v_0$ is rapidly decreasing then so does the solution itself and
\begin{equation*}\label{decreas}
\stred\sup_{t\in[0,T]}\bigg(\int_{\mr^d}|v(t,y)|(1+|y|)^n\,\dif y\bigg)^p<\infty,\qquad \forall n\in\mn_0,\,p\in[1,\infty).
\end{equation*}

The choice of the Stratonovich integral is more natural here and is given by application of the It\^o-Wentzell-type formula in the proof of the explicit representation of the solution \eqref{explic}. Indeed, in this case it is close to the classical differential rule formula for composite functions (cf. \cite[Theorem I.8.1, Theorem I.8.3]{kunita}).

\section{Solution to the stochastic BGK model}
\label{existence}

This section is devoted to the existence proof for the stochastic BGK model \eqref{bgk}. Let us start with the definition of its solution.
% Recall, that $\mathcal{P}$ denotes the predictable $\sigma$-algebra on $\Omega\times[0,T]$ associated to $(\mf_t)_{t\geq 0}$. In the following, we write $L^\infty_{\mathcal{P}}(\Omega\times[0,T]\times\mt^N\times\mr)$ to indicate predictability, i.e. to shorten the notation of
%$$L^\infty\big(\Omega\times[0,T]\times\mt^N\times\mr,\mathcal{P}\otimes\mathcal{B}(\mt^N)\otimes\mathcal{B}(\mr),\dif\prst\otimes\dif t\otimes\dif x\otimes\dif \xi\big).$$
%underlying measure spaces of various $L^\infty$-spaces need to be specified, however, as the corresponding measures are always obvious we will simplify the notation and write only the $\sigma$-algebras.

%some definitions. First, we need to define a space of functions where \eqref{bgk} will be considered. Set
%\begin{equation*}
%\begin{split}
%\mathcal{Y}&=\big\{Y\in L^\infty(\Omega\times[0,T],\mathcal{P},\dif \prst\otimes\dif t;L^\infty(\mt^N\times\mr));\\
%&\qquad\qquad  Y-\ind_{0>\xi}\in L^1(\Omega\times[0,T]\times\mt^N\times\mr)\big\}.
%\end{split}
%\end{equation*}

\begin{defin}
Let $\varepsilon>0$.
% Assume that $F_0^\varepsilon\in L^\infty(\Omega\times\mt^N\times\mr,\mf_0\otimes\mathcal{B}(\mt^N\times\mr))$ satisfies $F_0^\varepsilon-\ind_{0>\xi}\in L^1(\Omega\times\mt^N\times\mr)$. 
Then $F^\varepsilon\in L^\infty_{\mathcal{P}}(\Omega\times[0,T]\times\mt^N\times\mr)$
satisfying $F^\varepsilon-\ind_{0>\xi}\in L^1(\Omega\times[0,T]\times\mt^N\times\mr)$ is called a weak solution to the stochastic BGK model \eqref{bgk} with initial condition $F_0^\varepsilon$ provided the following holds true for a.e. $t\in[0,T]$, $\prst$-a.s.,
\begin{equation*}
\begin{split}
&\quad\big\langle F^\varepsilon(t),\varphi\big\rangle=\big\langle F^\varepsilon_0,\varphi\big\rangle+\int_0^t\big\langle F^\varepsilon(s),a\cdot\nabla\varphi\big\rangle\,\dif s\\
&\hspace{-.5cm}+\frac{1}{\varepsilon}\int_0^t\big\langle\ind_{u^\varepsilon(t)>\xi}-F^\varepsilon(t),\varphi(t)\big\rangle\,\dif t+\sum_{k=1}^d\int_0^t\big\langle F^\varepsilon(s),\partial_\xi(g_k\varphi)\big\rangle\,\dif \beta_k(s)\\
&\qquad+\frac{1}{2}\int_0^t\big\langle F^\varepsilon(s),\partial_\xi(G^2\partial_\xi\varphi)\big\rangle\,\dif s.
\end{split}
\end{equation*}
\end{defin}

\begin{rem}
In particular, for any $\varphi\in C^\infty_c(\mt^N\times\mr)$, there exists a representative of $\langle F^\varepsilon(t),\varphi\rangle\in L^\infty(\Omega\times[0,T])$ which is a continuous stochastic process.
\end{rem}

In order to solve the stochastic BGK model \eqref{bgk}, we intend to employ the stochastic characterics method introduced in the previous section. Hence we need to reformulate the problem in Strato\-novich form. It will be seen from the following lemma (see Corollary \ref{ssst}) that on the level of above defined weak solutions the problem \eqref{bgk} is equivalent to
\begin{equation*}
\begin{split}
\dif F^\varepsilon+a(\xi)\cdot\nabla F^\varepsilon\,\dif t&=\frac{\ind_{u^\varepsilon>\xi}-F^\varepsilon}{\varepsilon}\,\dif t-\partial_\xi F^\varepsilon\varPhi\circ\dif W+\frac{1}{4}\partial_\xi F^\varepsilon\partial_\xi G^2\,\dif t,\\
F^\varepsilon(0)&=F^\varepsilon_0.
\end{split}
\end{equation*}

\begin{lemma}\label{prevod}
If $X$ be a $C^1(\mt^N\times\mr)$-valued continuous $(\mf_t)$-semimartingale whose martingale part is given by $-\int_0^t\partial_\xi X\varPhi\,\dif W$, then
\begin{equation}\label{huhu}
-\int_0^t\!\partial_\xi X\varPhi\,\dif W+\frac{1}{2}\int_0^t\!\partial_\xi\big(G^2\partial_\xi X\big)\dif t=-\int_0^t\!\partial_\xi X\varPhi\circ\dif W+\frac{1}{4}\int_0^t\!\partial_\xi X \partial_\xi G^2\dif t.
\end{equation}
Moreover, the same is valid in the sense of distributions as well: let $X$ be a $\mathcal{D}'(\mt^N\times\mr)$-valued continuous $(\mf_t)$-semimartingale whose martingale part is given by $-\int_0^t\partial_\xi X\varPhi\,\dif W$, i.e. $\langle X(t),\varphi\rangle$ is a continuous $(\mf_t)$-semimartingale with martingale part $-\int_0^t\langle\partial_\xi X\varPhi,\varphi\rangle\,\dif W$ for any $\varphi\in C^\infty_c(\mt^N\times\mr)$. Then \eqref{huhu} holds true in $\mathcal{D}'(\mt^N\times\mr)$.

\begin{proof}
We will only prove the second part of the statement as the first one is straightforward and follows similar arguments.
Let us recall the relation between It\^{o} and Stratonovich integrals (see \cite{kunita} or \cite{kun1}). Let $Y$ be a continuous local semimartingale and $\Psi$ be a continuous semimartingale. Then the Stratonovich integral is well defined and satisfies
\begin{equation*}\label{relation}
\int_0^t\Psi\circ\dif Y=\int_0^t\Psi\,\dif Y+\frac{1}{2}\langle\!\langle\Psi,Y\rangle\!\rangle_t,
\end{equation*}
where $\langle\!\langle\cdot,\cdot\rangle\!\rangle_t$ denotes the cross-variation process. Therefore, we need to calculate the cross variation of $-\partial_\xi X g_k$ and the Wiener process $\beta_k$, $k=1,\dots,\,d$. Towards this end, we take a test function $\varphi\in C^\infty_c(\mt^N\times\mr)$ and derive the martingale part of $\langle \partial_\xi X g_k,\varphi\rangle$ (in the following, we emphasize only the corresponding martingale parts).
\begin{equation*}
\begin{split}
\langle X,\varphi\rangle&=\dots-\int_0^t\big\langle\partial_\xi X g_k,\varphi\big\rangle\,\dif \beta_k(s),\\
\langle X,g_k\varphi\rangle&=\dots-\int_0^t\big\langle\partial_\xi X g_k,g_k\varphi\big\rangle\,\dif \beta_k(s),\\
\langle\partial_\xi X, g_k\varphi\rangle&=\dots+\int_0^t\big\langle\partial_\xi X g_k,\partial_\xi(g_k\varphi)\big\rangle\,\dif \beta_k(s),
\end{split}
\end{equation*}
where
\begin{equation*}
\begin{split}
\big\langle\partial_\xi X g_k,\partial_\xi(g_k\varphi)\big\rangle&=-\big\langle\partial_\xi(\partial_\xi X g_k),g_k\varphi\big\rangle\\
&=-\big\langle\partial^2_\xi X g_k^2,\varphi\big\rangle-\frac{1}{2}\big\langle\partial_\xi X\partial_\xi g_k^2,\varphi\big\rangle\\
&=-\big\langle\partial_\xi(g_k^2\partial_\xi X),\varphi\big\rangle+\frac{1}{2}\big\langle\partial_\xi X \partial_\xi g_k^2,\varphi\big\rangle.
\end{split}
\end{equation*}
Consequently
\begin{equation*}
\begin{split}
\big\langle\!\!\big\langle\langle -\partial_\xi X g_k,\varphi\rangle,\beta_k\big\rangle\!\!\big\rangle_t=\int_0^t\big\langle\partial_\xi (g_k^2\partial_\xi X),\varphi\big\rangle\,\dif s-\frac{1}{2}\int_0^t\big\langle\partial_\xi X\partial_\xi g_k^2,\varphi\big \rangle\,\dif s
\end{split}
\end{equation*}
and the claim follows by summing up over $k$.
\end{proof}
\end{lemma}

\begin{cor}\label{ssst}
Let $\varepsilon>0$. If $F^\varepsilon\in L^\infty_{\mathcal{P}}(\Omega\times[0,T]\times\mt^N\times\mr)$
is such that $F^\varepsilon-\ind_{0>\xi}\in L^1(\Omega\times[0,T]\times\mt^N\times\mr)$ then it is a weak solution to \eqref{bgk} if and only if, for any $\varphi\in C^\infty_c(\mt^N\times\mr)$, there exists a representative of $\langle F^\varepsilon(t),\partial_\xi(g_k\varphi)\rangle\in L^\infty(\Omega\times[0,T])$ which is a continuous $(\mf_t)$-semimartingale and the following holds true for a.e. $t\in[0,T]$, $\prst$-a.s.,
\begin{equation*}
\begin{split}
&\quad\big\langle F^\varepsilon(t),\varphi\big\rangle=\big\langle F^\varepsilon_0,\varphi\big\rangle+\int_0^t\big\langle F^\varepsilon(s),a\cdot\nabla\varphi\big\rangle\,\dif s\\
&\hspace{-.5cm}+\frac{1}{\varepsilon}\int_0^t\big\langle\ind_{u^\varepsilon(t)>\xi}-F^\varepsilon(t),\varphi(t)\big\rangle\,\dif t+\sum_{k=1}^d\int_0^t\big\langle F^\varepsilon(s),\partial_\xi(g_k\varphi)\big\rangle\circ\dif \beta_k(s)\\
&\qquad-\frac{1}{4}\int_0^t\big\langle F^\varepsilon(s),\partial_\xi (\varphi \partial_\xi G^2)\big\rangle\,\dif s.
\end{split}
\end{equation*}
\end{cor}

%It is taken from \cite[Theorem I.8.3]{kunita}.
%\begin{thm}[It\^o-Wentzell formula]\label{wentzell}
%Let $W$ be a {\color{red}(???infinite-dimensional???)} Wiener process. Suppose that $F$ is a process satisfying
%$$\dif F(t,x)= a(t,x)\dif t+ A(t,x)\circ\dif W,\qquad x\in\mr^N,\,t\in[0,T],$$
%where $a(t,\cdot),A(t,\cdot)\in$ hladkost.
%and that $Y$ is a continuous semimartingale. Then the following formula holds true
%\begin{equation}\label{ito}
%\dif F(t,Y(t))=a(t,Y(t))\dif t+ A(t,Y(t))\circ\dif W+\sum_{i=1}^N\partial_{x_i}F(t,Y(t))\circ\dif Y^i(t).
%\end{equation}
%\end{thm}

As the first step in order to show the existence of a solution to the stochastic BGK model, we shall study the following auxiliary problem:
\begin{equation}\label{rov1}
\begin{split}
\dif X+a(\xi)\cdot\nabla X\,\dif t&=-\partial_\xi X\varPhi\circ\dif W+\frac{1}{4}\partial_\xi X\partial_\xi G^2\,\dif t,\\
X(s)&=X_0.
\end{split}
\end{equation}
It will be shown in Corollary \ref{weaksol1} that this problem possesses a unique weak solution provided $X_0\in  L^\infty(\Omega\times\mt^N\times\mr)$. Let
$$\mathcal{S}=\{\mathcal{S}(t,s);\,0\leq s\leq t\leq T\}$$
be its solution operator, i.e. for any $0\leq s\leq t\leq T$ we define $\mathcal{S}(t,s)X_0$ to be the solution to \eqref{rov1}. Then we have the following existence result for the stochastic BGK model.

\begin{thm}\label{duhamel}
For any $\varepsilon>0$, there exists a unique weak solution of the stochastic BGK model \eqref{bgk} and is represented by
\begin{equation}\label{sol1}
F^\varepsilon(t)=\me^{-\frac{t}{\varepsilon}}\mathcal{S}(t,0)F_0^\varepsilon+\frac{1}{\varepsilon}\int_0^t\me^{-\frac{t-s}{\varepsilon}}\mathcal{S}(t,s)\ind_{u^\varepsilon(s)>\xi}\,\dif s.
\end{equation}
%For any $\varepsilon>0$, the stochastic BGK model \eqref{bgk1}, given equivalently by \eqref{rovnice}, admits a unique solution that belongs to $L^\infty(0,T\semicol L^1(\Omega\times\mt^N\times\mr)$ and is expressed as
%\begin{equation}\label{sol}
%\begin{split}
%f^\varepsilon&(t,x,\xi)=\me^{-\frac{t}{\varepsilon}}\mathcal{S}(t,0)f_0^\varepsilon(x,\xi)+\frac{1}{\varepsilon}\int_0^t\me^{-\frac{t-s}{\varepsilon}}\mathcal{S}(t,s)\chi_{u^\varepsilon(s,x)}\,\dif s.
%\end{split}
%\end{equation}
\end{thm}

%\begin{cor}\label{velkyf}
%For any $\varepsilon>0$, the unique solution of the stochastic BGK model \eqref{bgk} is given by 
%\begin{equation}\label{sol1}
%F^\varepsilon(t,x,\xi)=\me^{-\frac{t}{\varepsilon}}\mathcal{S}(t,0)F_0^\varepsilon(x,\xi)+\frac{1}{\varepsilon}\int_0^t\me^{-\frac{t-s}{\varepsilon}}\mathcal{S}(t,s)\ind_{u^\varepsilon(s,x)>\xi}\,\dif s.
%\end{equation}
%
%\begin{proof}
%Let $0\leq s\leq t\leq T$ and $x\in\mt^N$ be fixed but otherwise arbitrary. Remark, that due to the assumption \eqref{nula} the process $\varphi^{R,0}_{s,t}(x,0)\equiv 0$ is a solution to the first equation in \eqref{char} for any $R>0$. Moreover, since the solution to \eqref{char} is unique, we deduce
%$$\varphi^{R,0}_{s,t}(x,\xi)\begin{cases}
%						\geq 0,&\quad \text{if }\;\xi\geq 0,\\
%						\leq 0,&\quad \text{if }\;\xi\leq 0.
%					\end{cases}$$
%As a consequence, $\mathcal{S}^R(t,s)\ind_{0>\xi}=\ind_{0>\xi}$ for all $R>0$ and adding $\ind_{0>\xi}$ to \eqref{sol} yields \eqref{sol1}.
%\end{proof}
%\end{cor}

The proof of Theorem \ref{duhamel} will be divided into several steps. First, we have to concentrate on the problem \eqref{rov1}.

\subsection{Application of the stochastic characteristics method}
\label{appl}

In this subsection, we prove the existence of a unique solution to \eqref{rov1} and study the behavior of the solution operator $\mathcal{S}$. The equation \eqref{rov1} is a first-order linear stochastic partial differential equation of the form \eqref{ch}, however, the coefficient $a$, as well as $\partial_\xi G^2$ in the case of \eqref{nula}, is not supposed to have bounded derivatives. For this purpose we introduce the following truncated problem: let $(k_R)$ be a smooth truncation on $\mr$, i.e. let $k_R(\xi)=k(R^{-1}\xi)$, where $k$ is a smooth function with compact support satisfying $0\leq k\leq 1$ and
$$k(\xi)=\begin{cases}
            1,& \text{if }\;|\xi|\leq\frac{1}{2},\\
	    0,& \text{if }\;|\xi|\geq1,
           \end{cases}
$$
and define
$g^R_k(x,\xi)=g_k(x,\xi)k_R(\xi),\,k=1,\dots, d$, and $a^R(\xi)=a(\xi)k_R(\xi)$. Coefficients $\varPhi^R$ and $G^{R,2}$, respectively, can be defined similarly as $\varPhi$ and $G^2$, respectively, using $g_k^R$ instead of $g_k$.\footnote{For notational simplicity we write $G^{R,2}$ as an abbreviation for $\big(G^R\big)^2$ and similarly $g_k^{R,2}$ instead of $\big(g_k^R\big)^2$.}
Then
\begin{equation}\label{rov10}
\begin{split}
\dif X+a^R(\xi)\cdot\nabla X\,\dif t&=-\partial_\xi X\varPhi^R\circ\dif W+\frac{1}{4}\partial_\xi X\partial_\xi G^{R,2}\,\dif t,\\
X(s)&=X_0
\end{split}
\end{equation}
can be solved by the method of stochastic characteristics. Indeed, the stochastic characteristic system associa\-ted with \eqref{rov10} is defined by the following system of Stratonovich's stochastic differential equations
\begin{equation}\label{char}
\begin{split}
\dif\varphi^0_t&=-\frac{1}{4}\partial_\xi G^{R,2}(\varphi_t)\,\dif t+\sum_{k=1}^dg^R_k(\varphi_t)\circ\dif\beta_k(t),\\
\dif\varphi^i_t&=a^R_i(\varphi^0_t)\,\dif t,\qquad i=1,\dots,N,\\
%\dif\eta_t&=-\frac{1}{4}\partial_\xi^2 G^2(\varphi_t)\eta_t\,\dif t,
\end{split}
\end{equation}
where the processes $\varphi_t^0$ and $\varphi_t^i,\,i=1,\dots,N,$ respectively, describe the evolution of the $\xi$-coordinate and $x^i$-coordinate, $\,i=1,\dots,N,$ respectively, of the characteristic curve.

%According to the classical theory for stochastic differential equations (see \cite{ito1}, \cite{ito2}), we conclude the existence of a unique solution to \eqref{char} provided the Lipschitz continuity hypothesis
%\begin{equation}\label{lipsch}
%\begin{split}
%|\partial_\xi G^2(x,\xi)-\partial_\xi G^2(y,\zeta)|^2+\sum_{k\geq 1}|g_k(x,\xi)&-g_k(y,\zeta)|^2+|a(\xi)-a(\zeta)|^2\\
%&\leq C\big(|x-y|^2+|\xi-\zeta|^2\big).
%\end{split}
%\end{equation}
%Note, that the above assumption is satisfied particularly for any linear diffusion coefficient independent of the space variable $x$, i.e. $g_k(x,\xi)=c_k\xi,\,k\in\mn$. Nevertheless, if the functions $c_k:\mt^N\rightarrow\mr,\,k\in\mn,$ are Lipschitz continuous and nonconstant, the first two conditions in \eqref{lipsch} holds true only locally. In that case, the classical theory only provides the existence of a local solution, up to the time of explosion. However, since the linear growth condition
%$$|\partial_\xi G^2(x,\xi)|^2+\sum_{k\geq 1}|g_k(x,\xi)|^2\leq C|\xi|^2$$
%holds true here, the explosion time is equal to infinity almost surely.
%Therefore, there exists a unique solution to the system \eqref{char} provided the local Lipschitz continuity hypothesis as well as the linear growth hypothesis is satisfied for the coefficients of the first two equations.
 
Let us denote by $\varphi^R_{s,t}(x,\xi)$ the solution of \eqref{char} starting from $(x,\xi)$ at time $s$. Then $\varphi^R$ defines a stochastic flow of $C^3$-diffeomorphisms and we denote by $\psi^R$ the corresponding inverse flow. It is the solution to the backward problem
\begin{equation}\label{char22}
\begin{split}
\dif\psi^0_t&=\frac{1}{4}\partial_\xi G^{R,2}(\psi_t)\,\hat{\dif\,}\! t-\sum_{k=1}^dg^R_k(\psi_t)\circ\hat{\dif\,}\!\beta_k(t),\\
\dif\psi^i_t&=-a^R_i(\psi^0_t)\,\hat{\dif\,}\!t,\qquad i=1,\dots,N.
\end{split}
\end{equation}

\begin{rem}
Note, that unlike the deterministic BGK model (i.e. $g_k=0,$ $k=1,\dots,\,d$), the stochastic case is not time homogeneous: $\varphi^R_{s,t}\neq\varphi^R_{0,t-s}$.
\end{rem}

\begin{prop}\label{aux}
Let $R>0$. If $X_0\in C^{3,\eta}(\mt^N\times\mr)$ almost surely,\footnote{$\eta>0$ is the H\"older exponent from Section \ref{setting}.} there exists a unique strong solution to \eqref{rov10} which is a continuous $C^{3,\vartheta}$-semimartingale for some $\vartheta>0$, i.e. it satisfies \eqref{rov10} in the following sense
\begin{equation*}
\begin{split}
X(t,x,\xi;s)&=X_0(x,\xi)-a^R(\xi)\cdot \int_s^t \nabla X(r,x,\xi;s)\,\dif r\\
&\qquad-\sum_{k=1}^d g_k^R(x,\xi)\int_s^t\partial_\xi X(r,x,\xi;s)\circ\dif \beta_k(r)\\
&\qquad+\frac{1}{4}\partial_\xi G^{R,2}(x,\xi)\int_s^t\partial_\xi X(r,x,\xi;s)\,\dif r,
\end{split}
\end{equation*}
Moreover, it is represented by
\begin{equation*}
%X(t,x,\xi;s)=\exp\bigg\{\!-\frac{1}{4}\int_0^t\!\partial^2_\xi G^2(\varphi_s(y,\zeta))\,\dif s\Big|_{(y,\zeta)=\psi_t(x,\xi)}\bigg\}X_0\big(\psi_t(x,\xi)\big)
X(t,x,\xi;s)=X_0\big(\psi^R_{s,t}(x,\xi)\big).
\end{equation*}

\begin{proof}
The above representation formula corresponds to \eqref{explic}. It can be shown in a straightforward manner using the It\^o-Wentzell formula (see \cite[Theorem 6.1.9]{kun1}).
\end{proof}
\end{prop}

%Later on, it will be useful to know more about the behavior of the solution operator to \eqref{rov1}
%$$\mathcal{S}(t,s)X_0=X(t,x,\xi;s)=X_0\big(\psi_{s,t}(x,\xi)\big),\qquad0\leq s\leq t\leq T.$$
%Towards this end, we introduce the following definition. It is a slight modification of the corresponding notion from \cite{kun1} which is however better adapted to our setting.
%
%\begin{defin}\label{df}
%A measurable function $h:\Omega\times\mt^N\times\mr\rightarrow\mr$ is called $\xi$-rapidly decreasing (uniformly in $x$) if, for any $n\in\mn_0$,
%\begin{equation*}\label{decay}
%\lim_{|\xi|\rightarrow\infty}\sup_{x\in\mt^N}|h(x,\xi)|(1+|\xi|)^n=0,\qquad\prst\text{-a.s.},
%\end{equation*}
%is satisfied.
%\end{defin}

%Let us denote by $\mathcal{Y}$ the subspace of $L^1(\Omega\times\mt^N\times\mr)$ containing all $\xi$-rapidly decreasing functions. Clearly, it is a linear space.

It is obvious, that the domain of definition of the solution operator to \eqref{rov10}, hereafter denoted by $\mathcal{S}^R$, can be extended to more general functions which do not necessarily fulfil the assumptions of Proposition \ref{aux}. In this case, we define consistently
$$\mathcal{S}^R(t,s)X_0=X_0\big(\psi^R_{s,t}(x,\xi)\big),\qquad 0\leq s\leq t\leq T.$$
Since diffeomorphisms preserve sets of measure zero the above is well defined also if $X_0$ is only defined almost everywhere. The resulting process cannot be a strong solution to \eqref{rov10}, however, as it will be seen in Corollary \ref{weaksol} it can still satisfy \eqref{rov10} in a weak sense. In the following proposition we establish basic properties of the ope\-rator $\mathcal{S}^R.$

\begin{prop}\label{oper}
Let $R>0$. Let $\mathcal{S}^R=\{\mathcal{S}^R(t,s),0\leq s\leq t\leq T\}$ be defined as above. Then
\begin{enumerate}
 \item\label{item1} $\mathcal{S}^R$ is a family of bounded linear operators on $L^1(\Omega\times\mt^N\times\mr)$ having unit operator norm, i.e. for any $X_0\in L^1(\Omega\times\mt^N\times\mr)$, $0\leq s\leq t\leq T$,
\begin{equation}\label{first} 
 \big\|\mathcal{S}^R(t,s)X_0\big\|_{L^1_{\omega,x,\xi}}\leq\|X_0\|_{L^1_{\omega,x,\xi}},
\end{equation}
\item\label{item4} $\mathcal{S}^R$ verifies the semigroup law
\begin{equation*}
\begin{split}
\mathcal{S}^R(t,s)&=\mathcal{S}^R(t,r)\circ\mathcal{S}^R(r,s),\qquad 0\leq s\leq r\leq t\leq T,\\
\mathcal{S}^R(s,s)&=\mathrm{Id},\hspace{3.79cm} 0\leq s\leq T.
\end{split}
\end{equation*}
\end{enumerate}

\begin{proof}
Fix arbitrary $0\leq s\leq t\leq T$. The linearity of $\mathcal{S}^R(t,s)$ follows easily from its definition.
In order to prove \eqref{first}, we will proceed in several steps. First, we make an additional assumption upon the initial condition $X_0$, namely,
\begin{equation}\label{addit}
X_0\in L^1(\Omega\times\mt^N\times\mr)\cap L^\infty(\Omega\times\mt^N\times\mr).
\end{equation}
Let us now consider a suitable smooth approximation of $X_0$. In particular, let $(h_\delta)$ be an approximation to the identity on $\mt^N\times\mr$, and $(k_\delta)$ a smooth truncation on $\mr$, i.e. define $k_\delta(\xi)=k(\delta\xi)$, where $k$ was defined at the beginning of this subsection.
Then the regularization $X_0^\delta$, defined in the following way
$$X_0^\delta(\omega)=\big(X_0(\omega)*h_\delta\big)k_\delta,$$
is bounded, pathwise smooth and compactly supported and
\begin{equation}\label{mm}
X_0^\delta\longrightarrow X_0\quad\text{in}\quad L^1(\Omega\times\mt^N\times\mr);\qquad \big\|X_0^\delta\big\|_{L^1_{\omega,x,\xi}}\leq\|X_0\|_{L^1_{\omega,x,\xi}}.
\end{equation}
Furthermore, also all the partial derivatives $\partial_\xi X_0^\delta,\,\partial_{x_i} X_0^\delta,\,i=1,\dots,\,N,$ are bounded, pathwise smooth and compactly supported.

Next, the process $X^\delta=\mathcal{S}^R(t,s)X_0^\delta$ is the unique strong solution to \eqref{rov10} or equivalently
\begin{equation}\label{rov20}
\begin{split}
\dif X+a^R(\xi)\cdot\nabla X\,\dif t&=-\partial_\xi X\varPhi^R\,\dif W+\frac{1}{2}\partial_\xi\big(G^{R,2}\partial_\xi X\big)\,\dif t,\\%\quad t\in[0,T],\,x\in\mt^N,\,\xi\in\mr,\\
X(s)&=X_0^\delta
\end{split}
\end{equation}
which follows by a similar approach as in Lemma \ref{prevod}.
For any $x\in\mt^N,\,\xi\in\mr$, the above stochastic integral is a well defined martingale with zero expected value. Indeed, for each $k=1,\dots,\,d$, we have\footnote{By $\nabla_{x,\xi}$ we denote the gradient with respect to the variables $x,\,\xi$.}
\begin{equation*}
\begin{split}
\stred\int_s^T\big|\partial_\xi X^\delta g^R_k(x,\xi)\big|^2\,\dif r&=C\,\stred\int_s^T\big|\nabla_{x,\xi}X_0^\delta\big(\psi^R_{s,r}(x,\xi)\big)\cdot\partial_\xi\psi^R_{s,r}(x,\xi)\big|^2\,\dif r\\
&\leq C\,\stred\int_s^T\big|\partial_\xi\psi^R_{s,r}(x,\xi)\big|^2\,\dif r<\infty
\end{split}
\end{equation*}
since $g_k^R$ is bounded and the process $\partial_\xi\psi^R_{s,r}(x,\xi)$ solves a backward bilinear stochastic differential equation with bounded coefficients (see \cite[Theorem 4.6.5]{kun1}) and therefore possesses moments of any order which are bounded in $0\leq s\leq r\leq T,\,x\in\mt^N,\,\xi\in\mr$. Nevertheless, we point out the same is not generally true without the assumption \eqref{addit}. In this case, the stochastic integral can happen to be a local martingale only, which would significantly complicate the subsequent steps.

We intend to integrate the equation \eqref{rov20} with respect to the variables $\omega,x,\xi$ and expect the stochastic integral to vanish. Towards this end, it is needed to verify the interchange of integrals with respect to $x,\,\xi$ and the stochastic one. We make use of the stochastic Fubini theorem \cite[Theorem 4.18]{daprato}. In order to verify its assumptions, the following quantity
\begin{equation*}
\begin{split}
\int_{\mt^N}&\int_\mr\bigg(\stred\int_s^T\big|\partial_\xi X^\delta g^R_k(x,\xi)\big|^2\,\dif r\bigg)^\frac{1}{2}\dif \xi\,\dif x\\
&\!=\int_{\mt^N}\int_\mr|g^R_k(x,\xi)|\bigg(\stred\int_s^T\big|\nabla_{x,\xi} X_0^\delta\big(\psi^R_{s,r}(x,\xi)\big)\cdot\partial_\xi\psi^R_{s,r}(x,\xi) \big|^2\,\dif r\bigg)^\frac{1}{2}\dif \xi\,\dif x
\end{split}
\end{equation*}
should be finite. Recall that $g^R_k,\,k=1,\dots,\,d,$ are bounded and the moments of $\partial_\xi\psi^R_{s,r}(x,\xi)$ are finite and bounded in $s,\,r,\,x,\,\xi.$ Thus, since $\nabla_{x,\xi}X_0^\delta$ is bounded and pathwise compactly supported it is sufficient to show that so does $\nabla_{x,\xi}X_0^\delta\big(\psi^R_{s,r}(x,\xi)\big)$. However, this fact follows immediately from the growth control on the stochastic flow $\psi^R$. Indeed, all the assertions of \cite[Section 4.5]{kun1}, in particular Exercise 4.5.9 and 4.5.10, can be modified in order to obtain corresponding results for the component $\psi^{R,0}_{s,r}$ only. Hence, for any $\eta\in(0,1)$, we have uniformly in $s,\,r,\,x,\,\prst$-a.s.,
\begin{equation*}\label{ada}
\lim_{|\xi|\rightarrow\infty}\frac{|\psi^{R,0}_{s,r}(x,\xi)|}{(1+|\xi|)^{1+\eta}}=0,\qquad\lim_{|\xi|\rightarrow\infty}\frac{(1+|\xi|)^\eta}{1+|\psi^{R,0}_{s,r}(x,\xi)|}=0.
\end{equation*}
Consequently, it yields: for any fixed $L>0$, there exists $l>0$ such that if $|\xi|>l$ then it holds uniformly in $s,\,r,\,x$, $\prst$-a.s.,
$$(1+|\xi|)^\eta\leq L(1+|\psi^{R,0}_{s,r}(x,\xi)|).$$
The support of $X_0^\delta$ as well as $\nabla_{x,\xi}X_0^\delta$ in the variable $\xi$ is included in $[-\frac{1}{\delta},\frac{1}{\delta}]$. Therefore, if in addition $(1+|\xi|)^\eta> L(1+\frac{1}{\delta})$ then $|\psi^{R,0}_{s,r}(x,\xi)|>\frac{1}{\delta}$ for all $s,\,r,\,x$, $\prst$-a.s., and accordingly $\nabla_{x,\xi}X_0^\delta\big(\psi^R_{s,r}(x,\xi)\big)=0$ for all $s,\,r,\,x$, $\prst$-a.s.. As a consequence, the stochastic Fubini theorem can be applied.

Therefore, integrating the equation \eqref{rov20} with respect to $\omega,x,\xi$ yields
\begin{equation*}
\begin{split}
&\stred\int_{\mt^N}\int_{\mr}X^\delta(t,x,\xi)\,\dif \xi\,\dif x+\stred\int_s^t\int_\mr a^R(\xi)\cdot\int_{\mt^N}\nabla X^\delta(r,x,\xi)\,\dif x\,\dif \xi\,\dif r\\
&\quad=\stred\int_{\mt^N}\int_{\mr}X^\delta_0\,\dif \xi\,\dif x+\frac{1}{2}\stred\int_s^t\int_{\mt^N}\int_\mr\partial_\xi\big(G^{R,2}(x,\xi)\partial_\xi X^\delta(r,x,\xi)\big)\,\dif \xi\,\dif x\,\dif r
\end{split}
\end{equation*}
where the second term on the left hand side vanishes due to periodic boundary conditions and the second one on the right hand side due to the compact support of $G^{R,2}$ in $\xi$. Hence we obtain
\begin{equation*}
\stred\int_{\mt^N}\int_{\mr}\mathcal{S}^R(t,s)X_0^\delta\,\dif \xi\,\dif x=\stred\int_{\mt^N}\int_{\mr}X^\delta_0\,\dif \xi\,\dif x
\end{equation*}
where the integrals on both sides are finite. Note, that if $X^\delta_0$ is nonnegative (nonpositive) then also $\mathcal{S}^R(t,s)X^\delta_0$ stays nonnegative (nonpositive). Therefore,
$$\big(\mathcal{S}^R(t,s)X_0^\delta\big)^+=\mathcal{S}^R(t,s)(X_0^\delta)^+,\qquad\big(\mathcal{S}^R(t,s)X_0^\delta\big)^-=\mathcal{S}^R(t,s)(X_0^\delta)^-,$$
and by splitting the initial data into positive and negative part we obtain that \eqref{first} is satisfied with equality in this case.

In addition to \eqref{mm}, also the convergence $\mathcal{S}^R(t,s)X_0^\delta\rightarrow\mathcal{S}^R(t,s)X_0$ holds true in $L^1(\Omega\times\mt^N\times\mr)$. Indeed, let us fix $\delta_1,\,\delta_2\in(0,1)$. Then \eqref{first} is also fulfilled by $X_0^{\delta_1}-X_0^{\delta_2}$ hence the set $\{\mathcal{S}^R(t,s)X_0^\delta;\,\delta\in(0,1)\}$ is Cauchy in $L^1(\Omega\times\mt^N\times\mr)$ and the limit is necessarily $\mathcal{S}^R(t,s)X_0$ since diffeomorphisms preserve sets of zero measure. Finally, application of the Fatou lemma gives \eqref{first} for $X_0$.

As the next step, we avoid the hypothesis \eqref{addit}. Let $X_0\in L^1(\Omega\times \mt^N\times\mr)$ and consider the following approximations
$$X_0^n=X_0\,\ind_{|X_0|\leq n},\qquad n\in\mn.$$
Then clearly
$$X_0^n\longrightarrow X_0\quad\text{in}\quad L^1(\Omega\times\mt^N\times\mr),\qquad\big\|X_0^n\big\|_{L^1_{\omega,x,\xi}}\leq \|X_0\|_{L^1_{\omega,x,\xi}}$$
and $X_0^n\in L^\infty(\Omega\times\mt^N\times\mr)$ hence the estimate \eqref{first} is valid for all $X_0^n$. As above, it is possible to show that $\mathcal{S}^R(t,s) X_0^n\rightarrow \mathcal{S}^R(t,s)X_0$ in $L^1(\Omega\times\mt^N\times\mr)$ and by the lower semicontinuity of the norm we obtain the claim.

Finally, item \ref{item4} can be shown by the flow property of $\psi$:
$$\mathcal{S}^R(t,r)\circ\mathcal{S}^R(r,s)X_0=X_0\big(\psi^R_{s,r}\big(\psi^R_{r,t}(x,\xi)\big)\big)=X_0\big(\psi^R_{s,t}(x,\xi)\big)=\mathcal{S}^R(t,s)X_0.$$
\end{proof}
\end{prop}

\begin{cor}\label{weaksol}
Let $R>0$. For any $\mf_{s}\otimes\mathcal{B}(\mt^N)\otimes\mathcal{B}(\mr)$-measurable initial datum $X_0\in L^\infty(\Omega\times\mt^N\times\mr)$ there exists a unique $X\in L^\infty_{\mathcal{P}_s}\big(\Omega\times[s,T]\times\mt^N\times\mr\big)$
that is a weak solution to \eqref{rov20}, i.e. the following holds true for any $\phi\in C_c^\infty(\mt^N\times\mr)$, a.e. $t\in[s,T]$, $\prst$-a.s.,
\begin{equation}\label{cd}
\begin{split}
&\big\langle X(t),\phi\big\rangle=\big\langle X_0,\phi\big\rangle+\int_s^t\big\langle X(r),a^R\cdot\nabla \phi\big\rangle\,\dif r\\
&\quad+\sum_{k=1}^d\int_s^t\big\langle X(r),\partial_\xi(g^R_k\phi)\big\rangle\,\dif \beta_k(r)+\frac{1}{2}\int_s^t\big\langle X(r),\partial_\xi(G^{R,2}\partial_\xi\phi)\big\rangle\,\dif r.
\end{split}
\end{equation}
Furthermore, it is represented by $X=\mathcal{S}^R(t,s)X_0$.
% and belongs to
%$$L^\infty(\Omega\times[s,T],\mathcal{P}_s,\dif\prst\otimes\dif t;L^\infty(\mt^N\times\mr)),$$
%where $\mathcal{P}_s$ denotes the predictable $\sigma$-algebra on $\Omega\times[s,T]$ associated to $(\mf_{s,t})_{t\geq0}$.

\begin{proof}
Let us start with the proof of uniqueness.
Due to linearity, it is enough to prove that any $L^\infty$-weak solution to \eqref{rov20} starting from the origin $X_0=0$ vanishes identically. Let $X$ be such a solution. First, let $(h_\tau)$ be a symmetric approximation to the identity on $\mt^N\times\mr$ and test \eqref{rov20} by $\phi(x,\xi)=h_\tau(y-x,\zeta-\xi)$. (Here, we employ the parameter $\tau$ in order to distinguish from the regularization defined in Proposition \ref{oper}, which will also be used in this proof.) Then $X^\tau(t):=X(t)*h_\tau$, for a.e. $t\in[s,T]$, satisfies
\begin{equation*}
\begin{split}
X^\tau(t,y,\zeta)&=-\int_s^t\big[a^R\cdot\nabla X(r)\big]^\tau(y,\zeta)\,\dif r-\sum_{k=1}^d\int_s^t\big[\partial_\xi X(r) g^R_k\big]^\tau(y,\zeta)\,\dif\beta_k(r)\\
&\qquad+\frac{1}{2}\int_s^t\big[\partial_\xi\big(G^{R,2}\partial_\xi X(r)\big)\big]^\tau(y,\zeta)\,\dif r
\end{split}
\end{equation*}
hence is smooth in $(y,\zeta)$ and can be extended to become continuous on $[s,T]$.
%\begin{equation*}
%\begin{split}
%X^\delta(t,y,\zeta)&=\int_s^t\big\langle X(r),a^R\cdot\nabla_x h_\delta(y-\cdot,\zeta-\cdot)\big\rangle\,\dif r\\
%&\qquad+\sum_{k=1}^d\int_s^t\big\langle X(r),\partial_\xi\big(g_k h_\delta(y-\cdot,\zeta-\cdot)\big)\big\rangle\,\dif\beta_k(r)\\
%&\qquad+\frac{1}{2}\int_s^t\big\langle X(r),\partial_\xi\big( G^2\partial_\xi h_\delta(y-\cdot,\zeta-\cdot)\big)\big\rangle\,\dif r.
%\end{split}
%\end{equation*}
Now, we will argue as in \cite[Theorem 20]{flandoli11} and make use of the stochastic flow $\varphi^R$. From the It\^o-Wentzell formula for the It\^o integral \cite[Theorem 3.3.1]{kun1} we deduce 
\begin{equation*}
\begin{split}
X^\tau\big(t,\varphi^R_{s,t}(\tilde{y},\tilde{\zeta})\big)&=-\int_s^t\big[a^R\cdot\nabla X(r)\big]^\tau\big(\varphi^R_{s,r}(\tilde{y},\tilde{\zeta})\big)\,\dif r\\
&\qquad-\sum_{k=1}^d\int_s^t\big[\partial_\xi X(r) g^R_k\big]^\tau\big(\varphi^R_{s,r}(\tilde{y},\tilde{\zeta})\big)\,\dif\beta_k(r)\\
&\qquad+\frac{1}{2}\int_s^t\big[\partial_\xi\big(G^{R,2}\partial_\xi X(r)\big)\big]^\tau\big(\varphi^R_{s,r}(\tilde{y},\tilde{\zeta})\big)\,\dif r\\
%=\int_s^t\big\langle X(r),a^R\cdot\nabla_x h_\delta\big(\varphi^R_{s,t}(\tilde{y},\tilde{\zeta})-\cdot\big)\big\rangle\,\dif r\\
%&\quad+\sum_{k=1}^d\int_s^t\big\langle X(r),\partial_\xi\big(g_k h_\delta\big(\varphi^R_{s,t}(\tilde{y},\tilde{\zeta})-\cdot\big)\big)\big\rangle\,\dif\beta_k(r)\\
%&\quad+\frac{1}{2}\int_s^t\big\langle X(r),\partial_\xi\big( G^2\partial_\xi h_\delta\big(\varphi^R_{s,t}(\tilde{y},\tilde{\zeta})-\cdot\big)\big)\big\rangle\,\dif r\\
&\qquad+\int_s^t\nabla X^\tau\big(r,\varphi^R_{s,r}(\tilde{y},\tilde{\zeta})\big)\cdot a^R\big(\varphi^{R,0}_{s,r}(\tilde{y},\tilde{\zeta})\big)\,\dif r\\
&\qquad+\sum_{k=1}^d\int_s^t\partial_\xi X^\tau\big(r,\varphi^R_{s,r}(\tilde{y},\tilde{\zeta})\big)g^R_k\big(\varphi^R_{s,r}(\tilde{y},\tilde{\zeta})\big)\,\dif \beta_k(r)\\
&\qquad+\frac{1}{2}\int_s^t\partial_\xi^2 X^\tau\big(r,\varphi^R_{s,r}(\tilde{y},\tilde{\zeta})\big) G^{R,2}\big(\varphi^R_{s,r}(\tilde{y},\tilde{\zeta})\big)\,\dif r\\
&\qquad-\sum_{k=1}^d\int_s^t\partial_\xi\big[\partial_\xi X(r) g^R_k\big]^\tau\big(\varphi^R_{s,r}(\tilde{y},\tilde{\zeta})\big) g^R_k\big(\varphi^R_{s,r}(\tilde{y},\tilde{\zeta})\big)\,\dif r\\
&=\mathrm{J}_1+\mathrm{J}_2+\mathrm{J}_3+\mathrm{J}_4+\mathrm{J}_5+\mathrm{J}_6+\mathrm{J}_7.
\end{split}
\end{equation*}
As the next step, we intend to show that $\mathrm{J}_1+\mathrm{J}_4\rightarrow 0,$ $\mathrm{J}_2+\mathrm{J}_5\rightarrow 0,$ and $\mathrm{J}_3+\mathrm{J}_6+\mathrm{J}_7\rightarrow 0$ in $\mathcal{D}'(\mt^N\times\mr)$, $\prst$-a.s., as $\tau\rightarrow 0$.
Remark, that unlike \cite{flandoli11}, working with the Stratonovich form of \eqref{rov20} would not bring any simplifications here. To be more precise, the Stratonovich version of the It\^o-Wentzell formula (see \cite[Theorem 3.3.2]{kun1}) is close to the classical differential rule formula for composite functions hence any correction terms (as $\mathrm{J}_6,\mathrm{J}_7$ in the It\^o version) are not necessary; however, due to the dependence on $x,\xi$ of the coefficients $g^R_k$, the corresponding Stratonovich integrals would not cancel and therefore in order to guarantee their convergence to zero, one would need to control the correction terms $\mathrm{J}_6,\mathrm{J}_7$ anyway.

Let us proceed with the proof of the above sketched convergence. Towards this end, we employ repeatedly the arguments of the commutation lemma of DiPerna and Lions \cite[Lemma II.1]{diperna}. In particular, in the case of $\mathrm{J}_1+\mathrm{J}_4$ we obtain for a.e. $r\in[s,t]$, $\prst$-a.s., that 
\begin{equation}\label{dur}
\begin{split}
a^R\cdot\nabla X^\tau(r)-\big[a^R\cdot\nabla X(r)\big]^\tau&\longrightarrow 0\quad\text{in}\quad\mathcal{D}'(\mt^N\times\mr).\\
%g^R_k\,\partial_\xi X^\delta(r)-\big[g^R_k\,\partial_\xi X(r)\big]^\delta&\longrightarrow 0,\quad\text{in}\quad\mathcal{D}'(\mt^N\times\mr),\\
%g^{R,2}_k\,\partial_\xi X^\delta(r)-\big[g^{R,2}_k\,\partial_\xi X(r)\big]^\delta&\longrightarrow 0,\quad\text{in}\quad\mathcal{D}'(\mt^N\times\mr).\\
\end{split}
\end{equation}
Indeed, since%if $X(r)$ is continuous in $(x,\xi)$, we obtain
\begin{equation*}
\begin{split}
&a^R(\xi)\cdot\nabla X^\tau(r,x,\xi)-\big[a^R\cdot\nabla X(r)\big]^\tau(x,\xi)\\
&\quad=\int_{\mt^N}\int_\mr X(r,y,\zeta)\big[a^R(\xi)-a^R(\zeta)\big]\cdot\nabla h_\tau(x-y,\xi-\zeta)\dif \zeta\dif y\\
%&=\int_{\mt^N}\int_\mr\int_0^1 X(r,y,\zeta)\totdif a^R\big(\zeta+\sigma(\xi-\zeta)\big)(\xi-\zeta)\cdot\nabla h_\delta(x-y,\xi-\zeta)\dif\sigma\dif \zeta\dif y\\
%&=\int_{\mt^N}\int_\mr\int_0^1 X\big(r,x-\delta\tilde{y},\xi-\delta\tilde{\zeta}\big)\totdif a^R\big(\xi-(1-\sigma)\delta\tilde{\zeta}\big)\tilde{\zeta}\cdot\nabla h(\tilde{y},\tilde{\zeta})\dif\sigma\dif \tilde{\zeta}\dif \tilde{y}\\
%&\longrightarrow X(r,x,\xi)\totdif a^R(\xi)\cdot\int_{\mt^N}\int_\mr\tilde{\zeta}\,\nabla h(\tilde{y},\tilde{\zeta})\dif\tilde{\zeta}\dif\tilde{y}=0
\end{split}
\end{equation*}
and $\tau|\nabla h_\tau|(\cdot)\leq C h_{2\tau}(\cdot)$, we obtain the following bound by standard estimates on convolutions : for any $\phi\in C^\infty_c(\mt^N\times\mr)$
\begin{equation*}%\label{dur1}
\begin{split}
\Big|\Big\langle a^R\cdot\nabla X^\tau(r)&-\big[a^R\cdot\nabla X(r)\big]^\tau,\phi\Big\rangle\Big|\\
&\leq C \big\|a^R\big\|_{W^{1,\infty}(\mr)}\|X(r)\|_{L^p(K_\phi)}\|\phi\|_{L^q(\mt^N\times\mr)},
\end{split}
\end{equation*}%}
where $K_\phi\subset\mt^N\times\mr$ is a suitable compact set and $p,q\in[1,\infty]$ are arbitrary conjugate exponents. As a consequence, it is sufficient to consider $X(r)$ continuous in $(x,\xi)$ as the general case can be concluded by a density argument. We have
\begin{equation*}
\begin{split}
%&a^R(\xi)\cdot\nabla X^\delta(r,x,\xi)-\big[a^R\cdot\nabla X(r)\big]^\delta(x,\xi)\\
&\int_{\mt^N}\int_\mr X(r,y,\zeta)\big[a^R(\xi)-a^R(\zeta)\big]\cdot\nabla h_\tau(x-y,\xi-\zeta)\dif \zeta\dif y\\
&=\int_{\mt^N}\int_\mr\int_0^1 X(r,y,\zeta)\totdif a^R\big(\zeta+\sigma(\xi-\zeta)\big)(\xi-\zeta)\cdot\nabla h_\tau(x-y,\xi-\zeta)\dif\sigma\dif \zeta\dif y\\
&=\int_{\mt^N}\int_\mr\int_0^1 X\big(r,x-\tau\tilde{y},\xi-\tau\tilde{\zeta}\big)\totdif a^R\big(\xi-(1-\sigma)\tau\tilde{\zeta}\big)\tilde{\zeta}\cdot\nabla h(\tilde{y},\tilde{\zeta})\dif\sigma\dif \tilde{\zeta}\dif \tilde{y}\\
&\longrightarrow X(r,x,\xi)\totdif a^R(\xi)\cdot\int_{\mt^N}\int_\mr\tilde{\zeta}\,\nabla h(\tilde{y},\tilde{\zeta})\dif\tilde{\zeta}\dif\tilde{y}=0
\end{split}
\end{equation*}
hence \eqref{dur} follows by the dominated convergence theorem.
Moreover, we deduce also that for a.e. $r\in[s,t]$, $\prst$-a.s.,
\begin{equation}\label{der1}
\begin{split}
a^R\big(\varphi^{R,0}_{s,r}\big)\cdot\nabla X^\tau\big(r,\varphi^R_{s,r}\big)-\big[a^R\cdot\nabla X(r)\big]^\tau\big(\varphi^R_{s,r}\big)&\longrightarrow 0\quad\text{in}\quad\mathcal{D}'(\mt^N\times\mr).\\
%g^R_k\big(\varphi^{R}_{s,r}\big)\,\partial_\xi X^\delta\big(r,\varphi^R_{s,r}\big)-\big[g^R_k\,\partial_\xi X(r)\big]^\delta\big(\varphi^R_{s,r}\big)&\longrightarrow 0,\quad\text{in}\quad\mathcal{D}'(\mt^N\times\mr),\\
%g^{R,2}_k\big(\varphi^{R}_{s,r}\big)\,\partial_\xi X^\delta\big(r,\varphi^R_{s,r}\big)-\big[g^{R,2}_k\,\partial_\xi X(r)\big]^\delta\big(\varphi^R_{s,r}\big)&\longrightarrow 0,\quad\text{in}\quad\mathcal{D}'(\mt^N\times\mr).\\
\end{split}
\end{equation}
It can be seen by using the change of variables formula: let $\mathrm{J}\psi^R_{s,r}$ denote the Jacobian of the inverse flow $\psi^R_{s,r}$, then
\begin{equation*}
\begin{split}
&\Big|\Big\langle a^R\big(\varphi_{s,r}^{R,0}\big)\cdot\nabla X^\tau\big(r,\varphi_{s,r}^{R}\big)-\big[a^R\cdot\nabla X(r)\big]^\tau\big(\varphi_{s,r}^{R}\big),\phi\Big\rangle\Big|\\
&=\Big|\Big\langle a^R\cdot\nabla X^\tau(r)-\big[a^R\cdot\nabla X(r)\big]^\tau,\phi\big(\psi_{s,r}^R\big)\big|\mathrm{J}\psi_{s,r}^R\big|\Big\rangle\Big|\\
&\leq C \big\|a^R\big\|_{W^{1,\infty}(\mr)}\|X(r)\|_{L^p(K)}\big\|\phi\big(\psi_{s,r}^R\big)\mathrm{J}\psi_{s,r}^R\big\|_{L^q(K)}\\
&\leq C \big\|a^R\big\|_{W^{1,\infty}(\mr)}\esssup_{s\leq r\leq T}\|X(r)\|_{L^p(K)}\|\phi\|_{L^\infty(K)}\sup_{s\leq r\leq T}\big\|\mathrm{J}\psi_{s,r}^R\big\|_{L^q(K)}<\infty,
\end{split}
\end{equation*}
which holds for a suitably chosen compact set $K\subset\mt^N\times\mr$ as $\phi(\psi_{s,r}^R)$ is compactly supported in $\mt^N\times\mr$ and any conjugate exponents $p,q\in[1,\infty]$. The estimate of $\sup_{s\leq r\leq T}\|\mathrm{J}\psi^R_{s,r}\|_{L^q(K)}$ is an immediate consequence of the fact that for almost every $\omega\in\Omega$ the mapping $(r,x,\xi)\mapsto \totdif \psi^R_{s,r}(\omega,x,\xi)$ is continuous due to the properties of stochastic flows (see Subsection \ref{flows}, \ref{itm:vik}) and therefore $(r,x,\xi)\mapsto \mathrm{J}\psi^R_{s,r}(\omega,x,\xi)$ is bounded on the given compact set $[s,T]\times K$. Having this bound in hand, we infer \eqref{der1} by using density again.
Accordingly, the almost sure convergence $\mathrm{J}_1+\mathrm{J}_4\rightarrow0$ in $\mathcal{D}'(\mt^N\times\mr)$ follows by the dominated convergence theorem.

In order to pass to the limit in the case of $\mathrm{J}_2+\mathrm{J}_5$, we employ the same approach as above so we will only write the main points of the proof. We obtain
%\begin{equation*}
%\begin{split}
%\Big|\Big\langle g^R_k\,\partial_\xi X^\delta(r)&-\big[g^R_k\,\partial_\xi X(r)\big]^\delta,\phi\Big\rangle\Big|\\
%&\leq C \big\|g_k^R\big\|_{W^{1,\infty}(\mr)}\|X(r)\|_{L^p(K_\phi)}\|\phi\|_{L^q(\mt^N\times\mr)}
%\end{split}
%\end{equation*}
%and
\begin{equation*}
\begin{split}
&\Big|\Big\langle g^R_k\big(\varphi_{s,r}^R\big)\,\partial_\xi X^\tau\big(r,\varphi_{s,r}^R\big)-\big[g^R_k\,\partial_\xi X(r)\big]^\tau\big(\varphi_{s,r}^R\big),\phi\Big\rangle\Big|\\
&\hspace{2mm}\leq C \big\|g_k^R\big\|_{W^{1,\infty}(\mr)}\esssup_{s\leq r\leq T}\|X(r)\|_{L^p(K)}\|\phi\|_{L^\infty(K)}\sup_{s\leq r\leq T}\big\|\mathrm{J}\psi_{s,r}^R\big\|_{L^q(K)}
\end{split}
\end{equation*}
hence for a.e. $r\in[s,T]$, $\prst$-a.s.,
\begin{equation*}
\begin{split}
g^R_k\big(\varphi^{R}_{s,r}\big)\,\partial_\xi X^\tau\big(r,\varphi^R_{s,r}\big)-\big[g^R_k\,\partial_\xi X(r)\big]^\tau\big(\varphi^R_{s,r}\big)&\longrightarrow 0\quad\text{in}\quad\mathcal{D}'(\mt^N\times\mr)\\
\end{split}
\end{equation*}
and accordingly we conclude by the dominated convergence theorem for stochastic integrals \cite[Theorem 32]{protter} that $\prst$-a.s. (up to subsequences) $\mathrm{J}_2+\mathrm{J}_5\rightarrow0$ in $\mathcal{D}'(\mt^N\times\mr)$.

Now, it remains to verify the convergence of $\mathrm{J}_3+\mathrm{J}_6+\mathrm{J}_7$. As the first step, we will show that for a.e. $r\in[s,T]$, $\prst$-a.s., in $\mathcal{D}'(\mt^N\times\mr)$
\begin{equation}\label{h1}
\begin{split}
\frac{1}{2}\big[\partial_\xi \big( g_k^{R,2}\partial_\xi X(r)\big)\big]^\tau+\frac{1}{2}\partial^2_{\xi\xi} X^\tau(r) g_k^{R,2}-\partial_\xi\big[\partial_\xi X(r) g^{R}_k\big]^\tau g^{R}_k\longrightarrow0.
\end{split}
\end{equation}
Towards this end, we observe
\begin{equation*}
\begin{split}
\frac{1}{2}\big[\partial_\xi \big( g_k^{R,2}\partial_\xi X(r)\big)\big]^\tau(x,\xi)&=\frac{1}{2}\big\langle\partial_\zeta X(r)g_k^{R,2},\partial_\xi h_\tau(x-\cdot,\xi-\cdot)\big\rangle,\\
\frac{1}{2}\partial^2_{\xi\xi} X^\tau(r,x,\xi) g_k^{R,2}(x,\xi)&=\frac{1}{2}\big\langle\partial_\zeta X(r),\partial_\xi h_\tau(x-\cdot,\xi-\cdot)\big\rangle g_k^{R,2}(x,\xi),\\
-\partial_\xi\big[\partial_\xi X(r) g^{R}_k\big]^\tau(x,\xi)\, g^{R}_k(x,\xi)&=-\big\langle\partial_\zeta X(r) g_k^R,\partial_\xi h_\tau(x-\cdot,\xi-\cdot)\big\rangle g_k^R(x,\xi),
\end{split}
\end{equation*}
and hence the left hand side of \eqref{h1} evaluated at $(x,\xi)$ is equal to
\begin{equation*}
\begin{split}
&\frac{1}{2}\int_{\mt^N}\int_\mr\partial_\zeta X(r,y,\zeta)\big(g_k^R(y,\zeta)-g_k^R(x,\xi)\big)^2\partial_\xi h_\tau(x-y,\xi-\zeta)\dif\zeta\dif y\\
&=-\int_{\mt^N}\int_\mr X(r,y,\zeta)\big(g_k^R(y,\zeta)-g_k^R(x,\xi)\big)\partial_\zeta g_k^R(y,\zeta)\partial_\xi h_\tau(x-y,\xi-\zeta)\dif\zeta\dif y\\
&\qquad+\frac{1}{2}\int_{\mt^N}\int_\mr X(r,y,\zeta)\big(g_k^R(y,\zeta)-g_k^R(x,\xi)\big)^2\partial^2_{\xi\xi}h_\tau(x-y,\xi-\zeta)\dif\zeta\dif y\\
&=\mathrm{I}_1(x,\xi)+\mathrm{I}_2(x,\xi).
\end{split}
\end{equation*}
Next, we proceed as in the case of $\mathrm{J}_1+\mathrm{J}_4$. We obtain
\begin{equation*}
\begin{split}
\big|\big\langle \mathrm{I}_1+\mathrm{I}_2,\phi\big\rangle\big|\leq C\big\|g_k^R\big\|_{W^{1,\infty}(\mt^N\times\mr)}^2\|X(r)\|_{L^p(K_\phi)}\|\phi\|_{L^q(\mt^N\times\mr)}
\end{split}
\end{equation*}
which holds true for a suitable compact set $K_\phi\subset\mt^N\times\mr$ and arbitrary conjugate exponents $p,q\in[1,\infty]$ and in the case of $X(r)$ continuous in $(x,\xi)$
\begin{equation*}
\begin{split}
\mathrm{I}_1(x,\xi)&\longrightarrow - X(r,x,\xi)\big(\partial_\xi g_k^R(x,\xi)\big)^2,\\
\mathrm{I}_2(x,\xi)&\longrightarrow X(r,x,\xi)\big(\partial_\xi g_k^R(x,\xi)\big)^2,
%=-\int_{\mt^N}\int_\mr\int_0^1 X(r,y,\zeta)\partial_\zeta g_k^R(y,\zeta)\totdif g_k(x+\sigma(y-x),\xi+\sigma(\zeta-\xi))\\
%&\hspace{3cm}\cdot(y-x,\zeta-\xi)\partial_\xi h_\delta(x-y,\xi-\zeta)\dif\sigma\dif\zeta\dif y\\
%&=\int_{\mt^N}\int_\mr\int_0^1 X(r,x-\delta)
\end{split}
\end{equation*}
which yields \eqref{h1} by the dominated convergence theorem and density.
As the next step, we conclude that
\begin{equation*}
\begin{split}
&\Big|\Big\langle \mathrm{I}_1\big(\varphi^R_{s,r}\big)+\mathrm{I}_2\big(\varphi^R_{s,r}\big),\phi\Big\rangle\Big|\\
&\leq C\big\|g_k^R\big\|_{W^{1,\infty}(\mt^N\times\mr)}^2\esssup_{s\leq r\leq T}\|X(r)\|_{L^p(K)}\|\phi\|_{L^\infty(K)}\sup_{s\leq r\leq T}\big\|\mathrm{J}\psi^R_{s,r}\big\|_{L^q(K)}
\end{split}
\end{equation*}
and consequently for a.e. $r\in[s,T]$, $\prst$-a.s.,
\begin{equation*}
\begin{split}
\frac{1}{2}\big[\partial_\xi \big( g_k^{R,2}\partial_\xi X(r)\big)\big]^\tau\big(\varphi^R_{s,r}\big)&+\frac{1}{2}\partial^2_{\xi\xi} X^\tau\big(r,\varphi^R_{s,r}\big) g_k^{R,2}\big(\varphi^R_{s,r}\big)\\
&\hspace{-2cm}-\partial_\xi\big[\partial_\xi X(r) g^{R}_k\big]^\tau\big(\varphi^R_{s,r}\big) g^{R}_k\big(\varphi^R_{s,r}\big)\longrightarrow0\quad\text{in}\quad\mathcal{D}'(\mt^N\times\mr).
\end{split}
\end{equation*}
Therefore, the desired convergence of $\mathrm{J}_3+\mathrm{J}_6+\mathrm{J}_7$ is verified.

Finally, since it holds true for a.e. $t\in[s,T]$ that
$$X^\tau(t)\overset{w^*}{\longrightarrow} X(t)\qquad\text{in}\qquad L^\infty(\mt^N\times\mr),\;\prst\text{-a.s.},$$
we obtain for any $\phi\in C^\infty_c(\mt^N\times\mr)$
\begin{equation*}
\begin{split}
\Big\langle X\big(t,\varphi_{s,t}^R\big),\phi\Big\rangle&=\Big\langle X(t), \phi\big(\psi^R_{s,t}\big)\big|\mathrm{J}\psi^R_{s,t}\big|\Big\rangle=\lim_{\tau\rightarrow 0}\Big\langle X^\tau(t), \phi\big(\psi^R_{s,t}\big)\big|\mathrm{J}\psi^R_{s,t}\big|\Big\rangle\\
&=\lim_{\tau\rightarrow 0}\Big\langle X^\tau\big(t,\varphi_{s,t}^R\big),\phi\Big\rangle=0
\end{split}
\end{equation*}
hence $X=0$ since $\varphi_{s,t}^R$ is a bijection and the proof of uniqueness is complete.

The proof of the explicit formula for $X$ follows by employing the regularization $X_0^\delta$ as in the proof of Proposition \ref{oper}.
The process $X^\delta=\mathcal{S}^R(t,s) X_0^\delta$ is the unique strong solution to \eqref{rov10} or equivalently \eqref{rov20} by using a similar approach as in Lemma \ref{prevod}. Consequently, it satisfies for all $\phi\in C^\infty_c(\mt^N\times\mr)$
\begin{equation*}
\begin{split}
\big\langle X^\delta(t),&\,\phi\big\rangle=\big\langle X^\delta_0,\phi\big\rangle+\int_s^t\big\langle X^\delta(r),a^R(\xi)\cdot\nabla \phi\big\rangle\,\dif r\\
&+\sum_{k=1}^d\int_s^t\big\langle X^\delta(r),\partial_\xi(g^R_k\phi)\big\rangle\,\dif \beta_k(r)+\frac{1}{2}\int_s^t\big\langle X^\delta(r),\partial_\xi(G^{R,2}\partial_\xi\phi)\big\rangle\,\dif r.
\end{split}
\end{equation*}
Now, it only remains to take the limit as $\delta\rightarrow 0$. As $X_0^\delta\rightarrow X_0$ for a.e. $\omega,x,\xi$ we have $X^\delta=\mathcal{S}^R(t,s)X_0^\delta\rightarrow \mathcal{S}^R(t,s)X_0=X$ for a.e. $\omega,x,\xi$ and every $t\in[s,T]$. Therefore, the convergence in all the terms apart from the stochastic one follows directly by the dominated convergence theorem.
For the case of stochastic integral we can apply the dominated convergence theorem for stochastic integrals. Since it holds
$$\big\langle X^\delta(r),\partial_\xi(g^R_k\phi)\big\rangle\longrightarrow\big\langle X(r),\partial_\xi(g^R_k\phi)\big\rangle,\qquad \text{a.e. }(\omega,\,r)\in\Omega\times[s,T]$$
and, setting $K=\supp \phi\subset\mt^N\times\mr$,
\begin{equation*}
\begin{split}
\big|\big\langle X^\delta(r),\partial_\xi(g^R_k\phi)\big\rangle\big|&\leq C\,\int_K\big|X_0^\delta\big(\psi^R_{s,r}(x,\xi)\big)\big|\,\dif\xi\,\dif x\leq C,
\end{split}
\end{equation*}
where the constant $C$ does not depend on $\delta$ due to the fact that
$$\|X^\delta_0\|_{L^\infty_{\omega,x,\xi}}\leq\|X_0\|_{L^\infty_{\omega,x,\xi}}.$$
Thus, we deduce (up to subsequences) the almost sure convergence of the stochastic integrals.
Furthermore, $\mathcal{S}^R(t,s)X_0$ is exactly the representative (in $t$) of the unique weak solution of \eqref{rov20} that satisfies \eqref{cd} for all $t\in[s,T]$, in particular, $t\mapsto\langle\mathcal{S}^R(t,s)X_0,\phi\rangle$ is a continuous $(\mf_{t})_{t\geq s}$-semimartingale for any $\phi\in C_c^\infty(\mt^N\times\mr)$.
\end{proof}
\end{cor}

%\begin{rem}\label{remar}
%According to the existence part of the proof of Corollary \ref{weaksol}, $\mathcal{S}^R(t,s)X_0$ is exactly the representative (in $t$) of the unique weak solution of \eqref{rov20} that satisfies \eqref{cd} for all $t\in[s,T]$, in particular, $\langle\mathcal{S}^R(t,s)X_0,\phi\rangle$ is a continuous semimartingale for any $\phi\in C_c^\infty(\mt^N\times\mr)$.
%\end{rem}

As the next step, we derive the existence of a unique weak solution to \eqref{rov1} which can be equivalently rewritten as
\begin{equation}\label{rov2}
\begin{split}
\dif X+a(\xi)\cdot\nabla X\,\dif t&=-\partial_\xi X\varPhi\,\dif W+\frac{1}{2}\partial_\xi(G^2\partial_\xi X)\,\dif t,\\
X(s)&=X_0
\end{split}
\end{equation}
due to Lemma \ref{prevod}.
With regard to the definition of the truncated coefficients, let us define
$$\tau^R(s,x,\xi)=\inf\big\{t\geq s\semicol\,|\varphi^{R,0}_{s,t}(x,\xi)|> R\big\}$$
(with the convention $\inf \emptyset=T$). Clearly, for any $s\in[0,T],\,x\in\mt^N,$ $\xi\in\mr$, $\tau^R(s,x,\xi)$ is a stopping time with respect to the filtration $(\mf_{t})_{t\geq s}$.
%$$\mf_{s,t}=\sigma\big(W_{r_1}-W_{r_2}\semicol\, s\leq r_1,r_2\leq t\big),\qquad 0\leq s\leq t\leq T.$$
Nevertheless, it can be shown that the blow-up cannot occur in a finite time and therefore
$$\sup_{R>0}\tau^R(s,x,\xi)=T,\qquad\prst\text{-a.s.},\,s\in[0,T],\,x\in\mt^N,\,\xi\in\mr.$$
Indeed, for any $R>0$,
the process $\varphi^{R,0}$ satisfies the It\^o equation
$$\dif \varphi^{R,0}_t=\sum_{k=1}^d g^R_k(\varphi^R_t)\,\dif\beta_k(t)$$
where all the coefficients $g_k^R$ satisfy the linear growth estimate \eqref{linrust} that is independent of $R$ and $x$ and therefore the claim follows by a standard estimation technique for SDEs.
%: for any $p\in[2,\infty)$,
%\begin{equation}\label{sde}
%\stred\sup_{s\leq t\leq T}\big|\varphi^{R,0}_{s,t}(x,\xi)\big|^p\leq C(1+|\xi|^p).
%\end{equation}
Moreover, if $R'>R$ then due to uniqueness $\tau^{R'}(s,x,\xi)\geq\tau^R(s,x,\xi)$ and $\mathcal{S}^{R'}(t,s)X_0=\mathcal{S}^R(t,s)X_0$ on $[0,\tau^R(s,x,\xi)]$. As a consequence, the pointwise limit
$$\big[\mathcal{S}(t,s)X_0\big](\omega,x,\xi):=\lim_{R\rightarrow\infty}\big[\mathcal{S}^R(t,s)X_0\big](\omega,x,\xi),\qquad 0\leq s\leq t\leq T,$$
exists almost surely and we obtain the following result.

\begin{cor}\label{weaksol1}
The family $\mathcal{S}=\{\mathcal{S}(t,s),\,0\leq s\leq t\leq T\}$ consists of bounded linear operators on $L^1(\Omega\times\mt^N\times\mr)$ having unit operator norm, i.e. for any $X_0\in L^1(\Omega\times\mt^N\times\mr)$, $0\leq s \leq t \leq T,$
$$\big\|\mathcal{S}(t,s)X_0\big\|_{L^1_{\omega,x,\xi}}\leq \|X_0\|_{L^1_{\omega,x,\xi}}.$$
Furthermore, for any $\mf_{s}\otimes\mathcal{B}(\mt^N)\otimes\mathcal{B}(\mr)$-measurable initial datum $X_0\in L^\infty(\Omega\times\mt^N\times\mr)$ there exists a unique $X\in L^\infty_{\mathcal{P}_s}(\Omega\times[s,T]\times\mt^N\times\mr)$ that is a weak solution to \eqref{rov2}. Besides, it is represented by $X=\mathcal{S}(t,s)X_0$ and $t\mapsto\langle\mathcal{S}(t,s)X_0,\phi\rangle$ is a continuous $(\mf_{t})_{t\geq s}$-semimartingale for any $\phi\in C_c^\infty(\mt^N\times\mr)$. Consequently, $\mathcal{S}$ verifies the semigroup law
\begin{equation*}
\begin{split}
\mathcal{S}(t,s)&=\mathcal{S}(t,r)\circ\mathcal{S}(r,s),\qquad 0\leq s\leq r\leq t\leq T,\\
\mathcal{S}(s,s)&=\mathrm{Id},\hspace{3.3cm} 0\leq s\leq T.
\end{split}
\end{equation*}

\begin{proof}
The first part of the proof follows directly from Proposition \ref{oper} while the rest is a consequence of Corollary \ref{weaksol}.
% To verify the rest, let us define regularizations $X_0^\delta$ of $X_0$ as in the proof of Proposition \ref{oper} and fix $R>0$ for the moment. It holds for any $0\leq s \leq t\leq T$, that
%$$X_0^\delta\longrightarrow X_0,\qquad \mathcal{S}^R(t,s)X_0^\delta\longrightarrow\mathcal{S}^R(t,s)X_0,\quad\qquad\text{a.e. }(\omega,x,\xi),$$
%which is enough to pass to the limit in the weak formulation of \eqref{rov20} hence $\mathcal{S}^R(t,s)X_0$ is the unique solution to \eqref{rov20}. Passing now to the limit as $R\rightarrow\infty$, the claim follows.
\end{proof}
\end{cor}

\begin{cor}\label{indik}
For all $n\in[0,\infty)$ it holds
\begin{equation}\label{indik1}
\sup_{0\leq s\leq T}\stred\sup_{s\leq t\leq T}\big\|\big(\mathcal{S}(t,s)\ind_{0>\xi}-\ind_{0>\xi}\big)(1+|\xi|)^n\big\|_{L^1_{x,\xi}}\leq C.
\end{equation}

\begin{proof}
Remark, that if \eqref{nula} is fulfilled, then for any $0\leq s\leq t\leq T$ and $x\in\mt^N$ the process $\varphi^{R,0}_{s,t}(x,0)\equiv 0$ is a solution to the first equation in \eqref{char} for any $R>0$. Moreover, since the solution to \eqref{char} is unique, we deduce
$$\varphi^{R,0}_{s,t}(x,\xi)\begin{cases}
						\geq 0,&\quad \text{if }\;\xi\geq 0,\\
						\leq 0,&\quad \text{if }\;\xi\leq 0.
					\end{cases}$$
As a consequence, the same is valid for the inverse stochastic flow $\psi^{R,0}$ hence $\mathcal{S}^R(t,s)\ind_{0>\xi}=\ind_{0>\xi}$ for all $R>0$ and thus the left hand side in \eqref{indik1} is zero.

In the case of \eqref{omez}, it is enough to prove the statement for any $\mathcal{S}^R$ provided the constant is independent on $R$. The stochastic characteristic system \eqref{char} rewritten in terms of It\^o's integral takes the following form
\begin{equation*}
\begin{split}
\dif \varphi^0_t&=\sum_{k=1}^d g^R_k(\varphi_t)\,\dif \beta_k(t),\\
\dif\varphi^i_t&=a_i^R(\varphi^0_t)\,\dif t,\qquad i=1,\dots,N,
\end{split}
\end{equation*}
whereas, in the case of the inverse flow, \eqref{char22} reads
\begin{equation*}
\begin{split}
\dif \psi^0_t&=-\sum_{k=1}^d g^R_k(\psi_t)\,\hat{\dif\,}\!\beta_k(t),\\
\dif\psi^i_t&=-a_i^R(\psi^0_t)\,\dif t,\qquad i=1,\dots,N.
\end{split}
\end{equation*}
Thus, we obtain
\begin{equation*}
\begin{split}
\mathcal{S}^R(t,s)\ind_{0>\xi}-\ind_{0>\xi}&=\ind_{\sum_{k=1}^d\int_s^t g^R_k(\psi^{R}_{r,t}(x,\xi))\hat{\dif\,}\!\beta_k(r)>\xi}-\ind_{0>\xi}\\
&=\ind_{|\xi|\leq\big|\sum_{k=1}^d\int_s^t g^R_k(\psi^{R}_{r,t}(x,\xi))\hat{\dif\,}\!\beta_k(r)\big|}\\
&\leq\frac{\big(1+\big|\sum_{k=1}^d\int_s^t g^R_k(\psi^{R}_{r,t}(x,\xi))\hat{\dif\,}\!\beta_k(r)\big|\big)^{n+2}}{(1+|\xi|)^{n+2}}
\end{split}
\end{equation*}
and since the fact that $\psi_{r,t}^{R}\circ\varphi_{s,t}^{R}=\varphi_{s,r}^{R}$ implies
\begin{equation*}
\begin{split}
\sum_{k=1}^d\int_s^t g^R_k(\psi_{r,t}^{R}(x,\xi))\,\hat{\dif\,}\!\beta_k(r)&=\sum_{k=1}^d\int_s^t g^R_k(\varphi_{s,r}^{R}(y,\zeta))\,\dif\beta_k(r)
\end{split}
\end{equation*}
by setting $(x,\xi)=\varphi_{s,t}^R(y,\zeta)$, we deduce that
\begin{equation*}
\begin{split}
\stred\sup_{s\leq t\leq T}&\int_{\mt^N}\int_\mr\big|\mathcal{S}(t,s)\ind_{0>\xi}-\ind_{0>\xi}\big|(1+|\xi|)^n\,\dif \xi\,\dif x\\
&\leq C+C\sup_{(y,\zeta)\in\mr^N\!\times\mr}\stred\sup_{s\leq t\leq T}\bigg|\sum_{k=1}^d\int_s^t g^R_k(\varphi^{R}_{s,r}(y,\zeta))\,\dif\beta_k(r)\bigg|^{n+2}\\
&\leq C+C\sup_{(y,\zeta)\in\mr^N\!\times\mr}\stred\bigg(\sum_{k=1}^d\int_s^T\big|g^R_k(\varphi^{R}_{s,r}(y,\zeta))\big|^2\,\dif r\bigg)^{\frac{n+2}{2}}\leq C,
\end{split}
\end{equation*}
where the constant $C$ does not depend on $R$ and $s$.
\end{proof}
\end{cor}

\begin{rem}\label{indikreason}
Let us make some comments on hypotheses \eqref{nula}, \eqref{omez} as the proof of Corollary \ref{indik} is their only use. The main difficulty in proving \eqref{indik1} comes from the unknown structure of dependence of the stochastic flows $\varphi^R$ and $\psi^R$ on $\xi$ in connection with the remaining variables $\omega,x,s,t$. Although one cannot say much in general, it is possible to find some (mostly simple) examples such that \eqref{indik1} holds true even without \eqref{nula}, \eqref{omez}. If the stochastic characteristic curve is governed by a linear system of stochastic differential equation as for instance
\begin{equation*}
\begin{split}
\dif \varphi^0_t&=\sum_{k=0}^N\big(1+\varphi^k_t\big)\,\dif \beta_k(t),\\
\dif\varphi^i_t&=\varphi^0_t\,\dif t,\qquad i=1,\dots, N,
\end{split}
\end{equation*}
i.e. neither \eqref{nula} nor \eqref{omez} is fulfilled since $g_0(x,\xi)=1+\xi$, then both forward and backward stochastic flow are given by explicit formulas where the dependence on $\xi$ is clear and, as a consequence, the statement of Corollary \ref{indik} remains valid.
%
%governing the $\xi$-coordinate of the characteristic curve
%\begin{equation*}\label{ex}
%\begin{split}
%\dif \varphi^0_t&=(1+\varphi^0_t)\,\dif \beta_1(t).\\
%%\psi^0(s)&=\xi.
%\end{split}
%\end{equation*}
%In this case $g_1(x,\xi)=1+\xi$ and therefore neither \eqref{nula} nor \eqref{omez} is fulfilled. Nevertheless, the reader can easily verify that the solution is given by
%$$\varphi^0_{s,t}(\xi)= Z(t,s)\bigg[\xi-\int_s^t\frac{1}{Z(r,s)}\,\dif r+\int_s^t\frac{1}{Z(r,s)}\,\dif \beta_1(r)\bigg],$$
%where
%$$Z(t,s)=\exp\bigg\{-\frac{1}{2}(t-s)+\beta_1(t)-\beta_1(s)\bigg\},$$
%hence a similar formula can be derived for the inverse flow. The dependence on $\xi$ is clear and, as a consequence, the statement of Corollary \ref{indik} remains valid.
\end{rem}

Now, we have all in hand to complete the proof of Theorem \ref{duhamel}.

\begin{proof}[Proof of Theorem \ref{duhamel}]
Recall, that the local densities are defined as follows
\begin{equation}\label{dens}
u^\varepsilon(t,x)=\int_\mr f^\varepsilon(t,x,\xi)\,\dif \xi=\int_\mr \big(F^\varepsilon(t,x,\xi)-\ind_{0>\xi}\big)\,\dif \xi
\end{equation}
hence the function $F^\varepsilon$ is not integrable with respect to $\xi$. For the purpose of the proof it is therefore more convenient to consider the process $h^\varepsilon(t)=F^\varepsilon(t)-\mathcal{S}(t,0)\ind_{0>\xi}$ instead and prove that it exists and is given by a suitable integral representation. Due to Corollary \ref{weaksol1}, $\mathcal{S}(t,s)\ind_{0>\xi}$ is the unique weak solution to \eqref{rov2} hence $h^\varepsilon$ solves
\begin{equation}\label{bgk4}
\begin{split}
\dif h^\varepsilon+a(\xi)\cdot\nabla h^\varepsilon\,\dif t&=\frac{(\ind_{u^\varepsilon>\xi}-\mathcal{S}(t,0)\ind_{0>\xi})-h^\varepsilon}{\varepsilon}\,\dif t-\partial_\xi h^\varepsilon\varPhi\,\dif W\\
&\qquad-\frac{1}{2}\partial_{\xi}\big(G^2(-\partial_\xi h^\varepsilon)\big)\,\dif t,\\
h^\varepsilon(0)&=\chi_{u_0^\varepsilon},
\end{split}
\end{equation}
in the sense of distributions.
%or equivalently
%\begin{equation}\label{bgk4strat}
%\begin{split}
%\dif h^\varepsilon+a(\xi)\cdot\nabla h^\varepsilon\,\dif t&=\frac{(\ind_{u^\varepsilon>\xi}-\mathcal{S}(t,0)\ind_{0>\xi})-h^\varepsilon}{\varepsilon}\,\dif t-\partial_\xi h^\varepsilon\varPhi\circ\dif W\\
%&\qquad+\frac{1}{4}\partial_{\xi}h^\varepsilon\partial_{\xi}G^2\,\dif t,\\
%h^\varepsilon(0)&=\chi_{u_0^\varepsilon}.
%\end{split}
%\end{equation}
Then, by Lemma \ref{prevod} and the weak version of Duhamel's principle, the problem \eqref{bgk4} admits an equivalent integral representation
\begin{equation}\label{sol}
h^\varepsilon(t)=\me^{-\frac{t}{\varepsilon}}\mathcal{S}(t,0)\chi_{u_0^\varepsilon}+\frac{1}{\varepsilon}\int_0^t\me^{-\frac{t-s}{\varepsilon}}\mathcal{S}(t,s)\big[\ind_{u^\varepsilon(s)>\xi}-\mathcal{S}(s,0)\ind_{0>\xi}\big]\dif s
\end{equation}
and thus can be solved by a fixed point method.
According to the identity
\begin{equation*}
\int_\mr|\ind_{\alpha>\xi}-\ind_{\beta>\xi}|\,\dif \xi=|\alpha-\beta|,\qquad \alpha,\,\beta\in\mr,
\end{equation*}
some space of $\xi$-integrable functions seems to be well suited to deal with the nonlinearity term $\ind_{u^\varepsilon>\xi}.$
%However, since
%$$F^\varepsilon=f^\varepsilon+\ind_{0>\xi},\qquad u^\varepsilon(t,x)=\int_\mr f^\varepsilon(t,x,\xi)\dif \xi,$$
%it is obvious that the function $F^\varepsilon$ is not integrable with respect to $\xi$. In order to overcome this difficulty we will, in fact, prove the existence of $f^\varepsilon$ as a fixed point of
%$$\big(\mathscr{K}g\big)(t,x,\xi)=\me^{-\frac{t}{\varepsilon}}\mathcal{S}(t,0)F^\varepsilon_0(x,\xi)+\frac{1}{\varepsilon}\int_0^t\me^{-\frac{t-s}{\varepsilon}}\mathcal{S}(t,s)\ind_{v(s,x)>\xi}\,\dif s-\ind_{0>\xi},$$
Let us denote $\mathscr{H}=L^\infty(0,T;L^1(\Omega\times\mt^N\times\mr))$ and show that the mapping
\begin{equation*}
\begin{split}
\big(\mathscr{K}g\big)&(t)=\me^{-\frac{t}{\varepsilon}}\mathcal{S}(t,0)\chi_{u_0^\varepsilon}+\frac{1}{\varepsilon}\int_0^t\me^{-\frac{t-s}{\varepsilon}}\mathcal{S}(t,s)\big[\ind_{v(s)>\xi}-\mathcal{S}(s,0)\ind_{0>\xi}\big]\dif s,
\end{split}
\end{equation*}
where the local density $v(s)=\int_\mr (g(s,\xi)+\mathcal{S}(s,0)\ind_{0>\xi}-\ind_{0>\xi})\dif \xi$ is defined consistently with \eqref{dens}, is a contraction on $\mathscr{H}$.
Let $g,\,g_1,\,g_2\in \mathscr{H}$ with corresponding densities $v,\,v_1,$ $v_2$. By Proposition \ref{oper}, Corollary \ref{indik} and the assumptions on initial data, we arrive at
\begin{equation*}
\begin{split}
\big\|(\mathscr{K}g)&(t)\big\|_{L^1_{\omega,x,\xi}}\\
&\leq\me^{-\frac{t}{\varepsilon}}\|\chi_{u_0^\varepsilon}\|_{L^1_{\omega,x,\xi}} +\frac{1}{\varepsilon}\int_0^t\me^{-\frac{t-s}{\varepsilon}}\|\ind_{v(s)>\xi}-\mathcal{S}(s,0)\ind_{0>\xi}\|_{L^1_{\omega,x,\xi}}\dif s\\
&\leq \|u_0^\varepsilon\|_{L^1_{\omega,x}}+\sup_{0\leq s\leq t}\Big(\|\chi_{v(s)}\|_{L^1_{\omega,x,\xi}}+\|\mathcal{S}(s,0)\ind_{0>\xi}-\ind_{0>\xi}\|_{L^1_{\omega,x,\xi}}\Big)\\
&\leq C+\sup_{0\leq s\leq t}\|g(s)\|_{L^1_{\omega,x,\xi}},
\end{split}
\end{equation*}
with a constant independent on $t$,
hence
\begin{equation*}
\begin{split}
\big\|\mathscr{K}g\big\|_{L^\infty_tL^1_{\omega,x,\xi}}&\leq C+\|g\|_{L^\infty_tL^1_{\omega,x,\xi}}<\infty.
\end{split}
\end{equation*}
Next, we have
\begin{equation*}
\begin{split}
\big\|(\mathscr{K} g_1)(t)-(\mathscr{K}g_2)(t)\big\|_{L^1_{\omega,x,\xi}}&\leq\frac{1}{\varepsilon}\int_0^t\me^{-\frac{t-s}{\varepsilon}}\|\ind_{v_1(s)>\xi}-\ind_{v_2(s)>\xi}\|_{L^1_{\omega,x,\xi}}\dif s\\
&= \frac{1}{\varepsilon}\int_0^t\me^{-\frac{t-s}{\varepsilon}}\|v_1(s)-v_2(s)\|_{L^1_{\omega,x}}\dif s\\
&\leq \frac{1}{\varepsilon}\int_0^t\me^{-\frac{t-s}{\varepsilon}}\|g_1(s)-g_2(s)\|_{L^1_{\omega,x,\xi}}\dif s,
\end{split}
\end{equation*}
so
$$\big\|\mathscr{K} g_1-\mathscr{K}g_2\big\|_{L^\infty_tL^1_{\omega,x,\xi}}\leq \big(1-\me^{-\frac{T}{\varepsilon}}\big)\|g_1-g_2\|_{L^\infty_tL^1_{\omega,x,\xi}}$$
and according to the Banach fixed point theorem, the mapping $\mathscr{K}$ has a unique fixed point in $\mathscr{H}$. Moreover, we deduce from Corollary \ref{weaksol1} that $h^\varepsilon$ is measurable with respect to $\mathcal{P}\otimes\mathcal{B}(\mt^N)\otimes\mathcal{B}(\mr)$ and therefore, according to the semigroup property of the solution operator $\mathcal{S}$, we obtain the existence of a unique weak solution to \eqref{bgk} that is expressed as \eqref{sol1} and the proof is complete.
% Corollary \ref{weaksol1}, $\mathcal{S}(t,0) \chi_{u_0^\varepsilon},\,\mathcal{S}(t,0) \ind_{0>\xi}$ and $\mathcal{S}(t,s)\ind_{u^\varepsilon(s)>\xi}$ are predictable
%$$\mathcal{S}(t,0) \chi_{u_0^\varepsilon},\,\mathcal{S}(t,0) \ind_{0>\xi}\in L^\infty_{\mathcal{P}}(\Omega\times[0,T]\times\mt^N\times\mr)$$ and $\mathcal{S}(t,s)\ind_{u^\varepsilon(s)>\xi}\in L^\infty_{\mathcal{P}_s}(\Omega\times[s,T]\times\mt^N\times\mr)$ 
\end{proof}

\begin{rem}
As a consequence of Corollary \ref{weaksol1}, it can be seen that the representative $h^\varepsilon(t)$ of the unique weak solution to \eqref{bgk4} that is given by \eqref{sol} satisfies: $t\mapsto\langle h^\varepsilon(t),\phi\rangle$ is a continuous $(\mf_t)$-semimartingale for any $\phi\in C_c^\infty(\mt^N\times\mr)$. Accordingly, $t\mapsto\langle F^\varepsilon(t),\phi\rangle$ is a continuous $(\mf_t)$-semimartingale for any $\phi\in C_c^\infty(\mt^N\times\mr)$ provided $F^\varepsilon(t)$ is the representative of the unique weak solution to \eqref{bgk} given by \eqref{sol1}.
\end{rem}

\subsection{Further properties of the solution operator}
\label{further}

In the previous subsection we showed that the family $\mathcal{S}$ consists of bounded linear operators on $L^1(\Omega\times\mt^N\times\mr)$ with unit operator norm which was essential for the existence proof for the stochastic BGK model in Theorem \ref{duhamel}. Nevertheless, for the proof of convergence of the BGK approximation in the next section, namely, to derive certain uniform estimates, we need to study also its behavior in other spaces. In particular, $\mathcal{S}(t,s)X_0$ is well defined if $X_0\in L^p(\Omega\times\mt^N\times\mr)$ and we obtain the following result.

\begin{prop}\label{pest}
For any $p\in[2,\infty)$, the family $\mathcal{S}$ consists of bounded li\-near operators on $L^p(\Omega\times\mt^N\times\mr)$ having unit operator norm. Moreover, the solution to \eqref{rov1} belongs to $L^p(\Omega;L^\infty(0,T;L^p(\mt^N\times\mr)))$ provided $X_0\in L^p(\Omega\times\mt^N\times\mr)$ and the following estimate holds true
\begin{equation}\label{sec1}
\sup_{0\leq s\leq T}\stred\sup_{s\leq t\leq T}\big\|\mathcal{S}(t,s)X_0\big\|_{L^p_{x,\xi}}^p\leq C\,\|X_0\|_{L^p_{\omega,x,\xi}}^p.
\end{equation}

\begin{proof}
Note, that it is enough to prove the statement for any $\mathcal{S}^R$ as the limit case of $\mathcal{S}$ then follows by Fatou lemma provided the constant in \eqref{sec1} does not depend on $R$. If $R>0$ is fixed then we use the same approach as in the proof of Proposition \ref{oper}, i.e. we will only prove the statement under the additional assumption
\begin{equation*}
X_0\in L^p(\Omega\times\mt^N\times\mr)\cap L^\infty(\Omega\times\mt^N\times\mr).
\end{equation*}
Let $X_0^\delta$ be bounded, pathwise smooth and compactly supported regularizations of $X_0$ such that
\begin{equation*}
X_0^\delta\longrightarrow X_0\quad\text{in}\quad L^p(\Omega\times\mt^N\times\mr),\qquad \big\|X_0^\delta\big\|_{L^p_{\omega,x,\xi}}\leq\|X_0\|_{L^p_{\omega,x,\xi}},
\end{equation*}
and $X^\delta=\mathcal{S}^R(t,s)X_0^\delta$ is the unique solution to \eqref{rov2}.
Now, we apply the It\^o formula to the function $h(v)=\|v\|_{L^p_{x,\xi}}^p$. If $q$ is the conjugate exponent to $p$ then $h'(v)=p|v|^{p-2}v\in L^q(\mt^N\times\mr)$ and
$$h''(v)=p(p-1)|v|^{p-2}\mathrm{Id}\in\mathscr{L}(L^p(\mt^N\times\mr);L^q(\mt^N\times\mr)).$$
Therefore
\begin{equation*}
\begin{split}
\big\|X^\delta&(t)\big\|^p_{L^p_{x,\xi}}=\big\|X_0^\delta\big\|^p_{L^p_{x,\xi}}\\
&-p\int_s^t\int_{\mt^N}\int_\mr\big|X^\delta\big|^{p-2}X^\delta\, a^R(\xi)\cdot\nabla X^\delta\,\dif \xi\,\dif x\,\dif r\\
&-p\sum_{k=1}^d\int_s^t\int_{\mt^N}\int_\mr\big|X^\delta\big|^{p-2}X^\delta\partial_\xi X^\delta g^R_k(x,\xi)\,\dif \xi\,\dif x\,\dif \beta_k(r)\\
&+\frac{p}{2}\int_s^t\int_{\mt^N}\int_\mr\big|X^\delta\big|^{p-2}X^\delta\,\partial_\xi\big(G^{R,2}\partial_\xi X^\delta\big)\,\dif \xi\,\dif x\,\dif r\\
&+\frac{p(p-1)}{2}\int_s^t\int_{\mt^N}\int_\mr\big|X^\delta\big|^{p-2}\big|\partial_\xi X^\delta\big|^2 G^{R,2}(x,\xi)\,\dif \xi\,\dif x\,\dif r.
\end{split}
\end{equation*}
Using integration by parts, the second term on the right hand side vanishes. Besides, having known the behavior of $X^\delta$ for large $\xi$, we integrate by parts in the fourth term and obtain the fifth term with opposite sign. To deal with the stochastic term, we also integrate by parts and observe
\begin{equation*}
\begin{split}
-p\int_\mr&\big|X^\delta\big|^{p-2}X^\delta\partial_\xi X^\delta g^R_k(x,\xi)\,\dif \xi\\
&=p(p-1)\int_\mr\big|X^\delta\big|^{p-2}\partial_\xi X^\delta X^\delta g^R_k(x,\xi)\,\dif \xi+p\int_{\mr}\big|X^\delta\big|^{p}\partial_\xi g^R_k(x,\xi)\,\dif\xi
\end{split}
\end{equation*}
hence
$$-p\int_\mr\big|X^\delta\big|^{p-2}X^\delta\partial_\xi X^\delta g^R_k(x,\xi)\,\dif \xi=\int_{\mr}\big|X^\delta\big|^{p}\partial_\xi g^R_k(x,\xi)\,\dif\xi$$
and we arrive at
\begin{equation*}
\begin{split}
\big\|X^\delta&(t)\big\|^p_{L^p_{x,\xi}}=\big\|X_0^\delta\big\|^p_{L^p_{x,\xi}}+\sum_{k=1}^d\int_s^t\int_{\mt^N}\int_\mr\big|X^\delta\big|^{p}\partial_\xi g^R_k(x,\xi)\,\dif \xi\,\dif x\,\dif \beta_k(r),
\end{split}
\end{equation*}
where the stochastic integral on the right hand side is a martingale with zero expected value.
Taking the expectation now yields
\begin{equation*}
 \stred\big\|X^\delta(t)\big\|^p_{L^p_{x,\xi}}=\stred\big\|X_0^\delta\big\|^p_{L^p_{x,\xi}}.
\end{equation*}
In order to derive \eqref{sec1}, we employ the Burkholder-Davis-Gundy inequality and boundedness of $\partial_\xi g_k$:
\begin{equation*}
\begin{split}
\stred\sup_{s\leq t\leq T}&\big\|X^\delta(t)\big\|_{L^p_{x,\xi}}^p\leq\stred\big\|X_0^\delta\big\|_{L^p_{x,\xi}}^p\\
&\hspace{.8cm}+\sum_{k=1}^d\stred \sup_{s\leq t\leq T}\int_s^t\int_{\mt^N}\int_\mr\big|X^\delta\big|^p\partial_\xi g^R_k(x,\xi)\,\dif \xi\,\dif x\,\dif\beta_k(r)\\
&\leq \stred\big\|X_0^\delta\big\|_{L^p_{x,\xi}}^p+C\,\stred\bigg(\int_s^T\big\|X^\delta(r)\big\|_{L^p_{x,\xi}}^{2p}\dif r\bigg)^\frac{1}{2}\\
&\leq \stred\big\|X_0^\delta\big\|_{L^p_{x,\xi}}^p+\frac{1}{2}\,\stred\sup_{s\leq t\leq T}\big\|X^\delta(t)\big\|_{L^p_{x,\xi}}^p+C\int_s^T\stred\big\|X^\delta(r)\big\|_{L^p_{x,\xi}}^p\,\dif r\\
\end{split}
\end{equation*}
hence
\begin{equation*}
\begin{split}
\stred\sup_{s\leq t\leq T}\big\|X^\delta(t)\big\|_{L^p_{x,\xi}}^p\leq C\,\stred\big\|X_0^\delta\big\|_{L^p_{x,\xi}}^p.
\end{split}
\end{equation*}
Note, that the constant $C$ does not depend on $\delta,s,R$. Therefore, the fact that the operator norm is equal to 1 as well as the validity of \eqref{sec1} follow easily by the same reasoning as in the proof of Proposition \ref{oper}.
\end{proof}
\end{prop}

\begin{prop}\label{chi}
Assume that $w\in L^p(\Omega\times\mt^N)$ for all $p\in[1,\infty)$. Then for all $n\in[0,\infty)$ there exists $r\in[1,\infty)$ such that
\begin{equation*}
\sup_{0\leq s\leq T}\stred\sup_{s\leq t\leq T}\big\|\big(\mathcal{S}(t,s)\chi_w\big)(1+|\xi|)^n\big\|_{L^1_{x,\xi}}\leq C\Big(1+\|w\|_{L^{r}_{\omega,x}}^{r}\Big),
\end{equation*}
where the constant $C$ does not depend on $w$.

\begin{proof}
We will prove that the claim holds true for all $\mathcal{S}^R$ with a constant independent of $R$. Let us denote by $\psi^{R,x}$ the vector of all $x^i$-coordinates of the stochastic flow $\psi^R$, i.e. $\psi_{s,t}^{R,x}(x,\xi)=\big(\psi_{s,t}^{R,1}(x,\xi),\dots,\psi_{s,t}^{R,N}(x,\xi)\big)$. Since it holds, for any $m\in[0,\infty)$,
\begin{equation*}
|\chi_w|\leq \frac{(1+|w|^2)^m}{(1+|\xi|^2)^m}\ind_{|\xi|<|w|}
\end{equation*}
we can estimate
\begin{equation}\label{plo1}
\begin{split}
\big|\mathcal{S}^R&(t,s)\chi_w\big|(1+|\xi|^n)=\big|\chi_{w(\psi_{s,t}^{R,x}(x,\xi))}(\psi^{R,0}_{s,t}(x,\xi))\big|(1+|\xi|)^n\\
&\,\,\leq \frac{\big(1+\big|w(\psi^{R,x}_{s,t}(x,\xi))\big|^2\big)^m}{(1+|\psi_{s,t}^{R,0}(x,\xi)|^2)^m}\ind_{|\psi^{R,0}_{s,t}(x,\xi)|<|w(\psi^{R,x}_{s,t}(x,\xi))|}(1+|\xi|)^n\\
&\,\,\leq\frac{(1+|\xi|^2)^{n/2}}{(1+|\psi_{s,t}^{R,0}(x,\xi)|^2)^m}\,\mathcal{S}^R(t,s)\Big[(1+|w|^2)^m\ind_{|\xi|<|w|}\Big],
\end{split}
\end{equation}
where the exact value of the exponent $m$ will be determined later on.
%The two fractions on the right hand side tend to zero uniformly in $s,\,t,\,x$ as $|\xi|\rightarrow\infty$ provided $n\leq m+\delta$. Moreover,
Now, we make use of the classical moment estimate for SDEs that in our setting reads
\begin{equation*}
\sup_{\substack{0\leq s\leq T\\ (y,\zeta)\in\mt^N\!\times\mr}}\stred\sup_{s\leq t\leq T}\frac{(1+|\varphi^{R,0}_{s,t}(y,\zeta)|^2)^p}{(1+|\zeta|^2)^p}\leq C,\qquad\forall p\in[1,\infty),
\end{equation*}
and rewritten in terms of the inverse flow by setting $(x,\xi)=\varphi_{s,t}^R(y,\zeta)$
%\cite[Lemma 4.5.3]{kun1} which in our setting reads {\color{red} je to vlastne obycejny odhad pro SDEs}
\begin{equation}\label{dud}
\sup_{\substack{0\leq s\leq T\\ (x,\xi)\in\mt^N\!\times\mr}}\stred\sup_{s\leq t\leq T}\frac{(1+|\xi|^2)^p}{(1+|\psi^{R,0}_{s,t}(x,\xi)|^2)^p}\leq C,\qquad\forall p\in[1,\infty),
\end{equation}
with a constant independent of $R$.
Therefore, employing \eqref{plo1}, the Young inequality, \eqref{dud} and Proposition \ref{pest} we obtain by a suitable choice of $m$
\begin{equation*}
\begin{split}
\sup_{0\leq s\leq T}&\stred\sup_{s\leq t\leq T}\int_{\mt^N}\int_\mr\big|\mathcal{S}^R(t,s)\chi_w\big|(1+|\xi|)^n\,\dif \xi\,\dif x\\
&\leq C\sup_{0\leq s\leq T}\stred\sup_{s\leq t\leq T}\int_{\mt^N}\int_\mr\frac{(1+|\xi|^2)^{n}}{(1+|\psi^{R,0}_{s,t}(x,\xi)|^2)^{2m}}\dif \xi\,\dif x\\
&\quad+C\sup_{0\leq s\leq T}\stred\sup_{s\leq t\leq T}\int_{\mt^N}\int_\mr\Big|\mathcal{S}^R(t,s)\Big[(1+|w|^2)^m\ind_{|\xi|<|w|}\Big]\Big|^2\dif \xi\,\dif x\\
%&\;= \mathrm{I}_1+\mathrm{I}_2
&\leq C+C\big\|(1+|w|^2)^m\ind_{|\xi|<|w|}\big\|_{L^{2}_{\omega,x,\xi}}^{2}\leq C\Big(1+\|w\|_{L^{4m+1}_{\omega,x}}^{4m+1}\Big)
\end{split}
\end{equation*}
%The H\"older inequality applied to the first term yields
%\begin{equation*}
%\begin{split}
%\sup_{0\leq s\leq T}\mathrm{I}_1&\leq C\sup_{0\leq s\leq T}\stred \sup_{s\leq t\leq T}\int_{\mt^N}\int_\mr\frac{(1+|\xi|^2)^{(n+2)q}}{(1+|\psi^{R,0}_{s,t}(x,\xi)|^2)^{2mq}}\,\dif \xi\,\dif x\\
%&\leq C\,\sup_{\substack{0\leq s\leq T\\(x,\xi)\in\mt^N\!\times\mr}}\stred\sup_{s\leq t\leq T}\frac{(1+|\xi|^2)^{(n+2)q+1}}{(1+|\psi^{R,0}_{s,t}(x,\xi)|^2)^{2mq}}\leq C,
%\end{split}
%\end{equation*}
%where the last inequality follows from \eqref{dud} by a suitable choice of $m$. Similarly for the second term we obtain by Proposition \ref{pest}
%\begin{equation*}
%\begin{split}
%\sup_{0\leq s\leq T}\mathrm{I}_2&\leq C\sup_{0\leq s\leq T}\stred\sup_{s\leq t\leq T}\Big\|\,\mathcal{S}^R(t,s)\Big[(1+|w|^2)^m\ind_{|\xi|<|w|}\Big]\Big\|_{L^{2q}_{x,\xi}}^{2q}\\
%&\leq C\big\|(1+|w|^2)^m\ind_{|\xi|<|w|}\big\|_{L^{2q}_{\omega,x,\xi}}^{2q}\leq C\Big(1+\|w\|_{L^{4mq+1}_{\omega,x}}^{4mq+1}\Big)
%\end{split}
%\end{equation*}
which completes the proof.
\end{proof}
\end{prop}

\section{Convergence of the BGK approximation}
\label{convergence}

In this final section, we investigate the limit of the stochastic BGK model as $\varepsilon\rightarrow 0$ and prove our main result, Theorem \ref{main}. To be more precise, we consider the following weak formulation of \eqref{bgk}, which is satisfied by $F^\varepsilon$, and show its convergence to the kinetic formulation of \eqref{conser}. Let $\varphi\in C_c^\infty([0,T)\times\mt^N\times\mr)$ then
\begin{equation}\label{formul}
\begin{split}
&\quad\int_0^T\big\langle F^\varepsilon(t),\partial_t\varphi(t)\big\rangle\,\dif t+\big\langle F_0^\varepsilon,\varphi(0)\big\rangle+\int_0^T\big\langle F^\varepsilon(t),a\cdot\nabla\varphi(t)\big\rangle\,\dif t\\
&=-\frac{1}{\varepsilon}\int_0^T\big\langle\ind_{u^\varepsilon(t)>\xi}-F^\varepsilon(t),\varphi(t)\big \rangle\,\dif t+\int_0^T\big\langle \partial_\xi F^\varepsilon(t)\,\varPhi\,\dif W(t),\varphi(t)\big\rangle\\
&\qquad\quad+\frac{1}{2}\int_0^T\big\langle G^2\partial_\xi F^\varepsilon(t),\partial_\xi \varphi(t)\big\rangle\,\dif t.
\end{split}
\end{equation}
A similar expression holds true also for $h^\varepsilon$, namely, it satisfies the weak formulation of \eqref{bgk4}. However, as in the following we restrict our attention to the representatives $F^\varepsilon(t)$ and $h^\varepsilon(t)$, respectively, given by \eqref{sol1} and \eqref{sol}, respectively, we point out that both are true even in a stronger sense.
For the case of $h^\varepsilon(t)$, we have: let $\varphi\in C^\infty_c(\mt^N\times\mr)$ then it holds for all $t\in[0,T]$
\begin{equation}\label{formul1}
\begin{split}
\big\langle h^\varepsilon(t),\varphi\big\rangle&=\big\langle h_0^\varepsilon,\varphi\big\rangle+\int_0^t\big\langle h^\varepsilon(s),a\cdot\nabla\varphi\big\rangle\,\dif s\\
&\hspace{-1cm}+\frac{1}{\varepsilon}\int_0^t\big\langle\ind_{u^\varepsilon(s)>\xi}-\mathcal{S}(s,0)\ind_{0>\xi}-h^\varepsilon(s),\varphi\big\rangle\,\dif s\\
&\quad-\int_0^t\big\langle \partial_\xi h^\varepsilon(s)\,\varPhi\,\dif W(s),\varphi\big\rangle-\frac{1}{2}\int_0^t\big\langle G^2\partial_\xi h^\varepsilon(s),\partial_\xi\varphi\big\rangle\,\dif s.
\end{split}
\end{equation}

\begin{proof}[Proof of Theorem \ref{main}]
Taking the limit in \eqref{formul} is quite straightforward in all the terms apart from the first one on the right hand side and can be done immediately. Remark, that according to the representation formula \eqref{sol1} it holds that the set of solutions $\{F^\varepsilon;\,\varepsilon\in(0,1)\}$ is bounded in $L^\infty_{\mathcal{P}}(\Omega\times[0,T]\times\mt^N\times\mr)$, more precisely, $F^\varepsilon\in[0,1],\,\varepsilon\in(0,1).$ Therefore, by the Banach-Alaoglu theorem, there exists $F\in L^\infty_{\mathcal{P}}(\Omega\times[0,T]\times\mt^N\times\mr)$ such that, up to subsequences, 
\begin{equation}\label{weakstar}
F^\varepsilon\overset{w^*}{\longrightarrow}F\quad\text{in}\quad L^\infty_{\mathcal{P}}(\Omega\times[0,T]\times\mt^N\times\mr).
\end{equation}
Hence, almost surely,
$$\int_0^T\big\langle F^\varepsilon(t),\,\partial_t\varphi(t)\big\rangle\,\dif t\longrightarrow\int_0^T\big\langle F(t),\,\partial_t\varphi(t)\big\rangle\,\dif t,$$
$$\int_0^T\big\langle F^\varepsilon(t),a\cdot\nabla\varphi(t)\big\rangle\,\dif t\longrightarrow\int_0^T\big\langle F(t),a\cdot\nabla\varphi(t)\big\rangle\,\dif t,$$
$$\frac{1}{2}\int_0^T\big\langle G^2\partial_\xi F^\varepsilon(t),\partial_\xi \varphi(t)\big\rangle\,\dif t\longrightarrow\frac{1}{2}\int_0^T\big\langle G^2\partial_\xi F(t),\partial_\xi \varphi(t)\big\rangle\,\dif t.$$
and, according to the hypotheses on the initial data,
$$\big\langle F_0^\varepsilon,\varphi(0)\big\rangle\longrightarrow\big\langle \ind_{u_0>\xi},\varphi(0)\big\rangle.$$
We intend to prove a similar convergence result for the stochastic term as well. Since
$$\big\langle F^\varepsilon,\partial_\xi(g_k\varphi)\big\rangle\longrightarrow\big\langle F,\partial_\xi(g_k\varphi)\big\rangle,\qquad \text{a.e. } (\omega,t)\in\Omega\times[0,T],$$
and, due to the boundedness of $F^\varepsilon$ and the assumptions on $g_k$,
$$\big|\big\langle F^\varepsilon,\partial_\xi(g_k\varphi)\big\rangle\big|\leq C,$$
the dominated convergence theorem for stochastic integrals gives (up to subsequences) the desired almost sure convergence
$$\int_0^T\big\langle \partial_\xi F^\varepsilon(t)\,\varPhi\,\dif W(t),\varphi(t)\big\rangle\longrightarrow\int_0^T\big\langle \partial_\xi F(t)\,\varPhi\,\dif W(t),\varphi(t)\big\rangle.$$
Furthermore, multiplying \eqref{formul} by $\varepsilon$ yields, almost surely,
\begin{equation}\label{klad}
\int_0^T\big\langle\ind_{u^\varepsilon(t)>\xi}-F^\varepsilon(t),\varphi(t)\big\rangle\,\dif t\longrightarrow 0
\end{equation}
and, in particular, 
\begin{equation}\label{lll}
\partial_\xi\ind_{u^\varepsilon>\xi}-\partial_\xi F^\varepsilon\longrightarrow 0
\end{equation}
in the sense of distributions over $(0,T)\times\mt^N\times\mr$ almost surely.
In order to obtain the convergence in the remaining term of \eqref{formul} and in view of the kinetic formulation of \eqref{conser}, we need to show that the term $\frac{1}{\varepsilon}(\ind_{u^\varepsilon>\xi}-F^\varepsilon)$ can be written as $\partial_\xi m^\varepsilon$ where $m^\varepsilon$ is a random nonnegative measure over $[0,T]\times\mt^N\times\mr$ bounded uniformly in $\varepsilon$. However, if we define
\begin{equation}\label{meas}
\begin{split}
m^\varepsilon(\xi)&=\frac{1}{\varepsilon}\int_{-\infty}^\xi \big(\ind_{u^\varepsilon>\zeta}-F^\varepsilon(\zeta)\big)\,\dif\zeta\\
&=\frac{1}{\varepsilon}\int_{-\infty}^\xi \big(\ind_{u^\varepsilon>\zeta}-\mathcal{S}(t,0)\ind_{0>\zeta}-h^\varepsilon(\zeta)\big)\,\dif\zeta,
\end{split}
\end{equation}
it is easy to check that $m^\varepsilon\geq 0$ since $F^\varepsilon\in[0,1]$. Indeed, $m^\varepsilon(-\infty)=m^\varepsilon(\infty)=0$ and $m^\varepsilon(t,x,\cdot)$ is increasing if $\xi\in(-\infty,u^\varepsilon(t,x))$ and decreasing if $\xi\in(u^\varepsilon(t,x),\infty)$.

Due to the convergence in \eqref{formul} it can be seen that for almost every $\omega\in\Omega$ there exists a distribution $m(\omega)$ such that, almost surely,
\begin{equation}\label{ddr}
\int_0^T\big\langle m^\varepsilon,\varphi(t)\big\rangle\,\dif t\longrightarrow \int_0^T\big\langle m,\varphi(t)\big\rangle\,\dif t,
\end{equation}
for any $\varphi\in C_c^\infty([0,T)\times\mt^N\times\mr)$. Besides, the conditions on test functions can be relaxed so that \eqref{ddr} holds true for any $\varphi\in C_c^\infty([0,T]\times\mt^N\times\mr)$.
Now, it remains to verify that $m$ is a kinetic measure. The following proposition will be useful.

\begin{prop}\label{densities}
The set of local densities $\{u^\varepsilon;\,\varepsilon\in(0,1)\}$ is bounded in $L^p(\Omega;L^\infty(0,T;L^p(\mt^N)))$ for all $p\in[1,\infty)$.

\begin{proof}
We need to find a uniform estimate for $u^\varepsilon$.
%According to \eqref{sol}, we have for $f^\varepsilon=F^\varepsilon-\ind_{0>\xi},\,F_0^\varepsilon=\ind_{u^\varepsilon_0>\xi},$
%\begin{equation}\label{sol1}
%\begin{split}
%f^\varepsilon&(t,x,\xi)=\me^{-\frac{t}{\varepsilon}}\mathcal{S}(t,0)\chi_{u_0^\varepsilon(x)}(\xi)+\frac{1}{\varepsilon}\int_0^t\me^{-\frac{t-s}{\varepsilon}}\mathcal{S}(t,s)\chi_{u^\varepsilon(s,x)}(\xi)\,\dif s\\
%%&+\me^{-\frac{t}{\varepsilon}}\Big(\mathcal{S}(t)\ind_{0>\xi}-\ind_{0>\xi}\Big)+\frac{1}{\varepsilon}\int_0^t\me^{-\frac{t-s}{\varepsilon}}\Big(\mathcal{S}(t-s)\ind_{0>\xi}-\ind_{0>\xi}\Big)\dif s
%\end{split}
%\end{equation}
It follows from the definition of $u^\varepsilon$ \eqref{dens} and \eqref{sol1} that
\begin{equation*}
\begin{split}
u^\varepsilon(t,x)&=\me^{-\frac{t}{\varepsilon}}\int_\mr\big(\mathcal{S}(t,0)\ind_{u_0^\varepsilon>\xi}-\ind_{0>\xi}\big)\,\dif\xi\\
&\qquad+\frac{1}{\varepsilon}\int_0^t\me^{-\frac{t-s}{\varepsilon}}\int_\mr\big(\mathcal{S}(t,s)\ind_{u^\varepsilon(s)>\xi}-\ind_{0>\xi}\big)\,\dif\xi\,\dif s.
\end{split}
\end{equation*}
Let us now define the following auxiliary function
$$H(s)=\left|\int_\mr\big(\mathcal{S}(t,s)\ind_{u^\varepsilon(s)>\xi}-\ind_{0>\xi}\big)\,\dif\xi\right|.$$
Then
$$H(t)\leq \me^{-\frac{t}{\varepsilon}} H(0)+(1-\me^{-\frac{t}{\varepsilon}})\max_{0\leq s\leq t} H(s)$$
and we conclude that $H(t)\leq H(0),\,t\in[0,T]$. In order to estimate $H(0)$, we make use of Proposition \ref{chi} and Corollary \ref{indik}. If $p=1$ they can be used directly
\begin{equation*}
\begin{split}
\stred\sup_{0\leq t\leq T}&\int_{\mt^N}|u^\varepsilon(t,x)|\,\dif x\leq \stred \sup_{0\leq t\leq T}\int_{\mt^N}\int_\mr\big|\mathcal{S}(t,0)\ind_{u^\varepsilon_0>\xi}-\ind_{0>\xi}\big|\,\dif \xi\,\dif x\\
&\leq \stred \sup_{0\leq t\leq T}\big\|\mathcal{S}(t,0)\chi_{u^\varepsilon_0}\big\|_{L^1_{x,\xi}}+ \stred \sup_{0\leq t\leq T}\big\|\mathcal{S}(t,0)\ind_{0>\xi}-\ind_{0>\xi}\big\|_{L^1_{x,\xi}}\\
&\leq C\Big(1+\|u_0^\varepsilon\|_{L^{r_1}_{\omega,x}}^{r_1}\Big),
\end{split}
\end{equation*}
whereas the case of $p\in(1,\infty)$ can be dealt with by the H\"older inequality and the fact that
$$\big|\mathcal{S}(t,0)\ind_{u^\varepsilon_0>\xi}-\ind_{0>\xi}\big|^p=\big|\mathcal{S}(t,0)\ind_{u^\varepsilon_0>\xi}-\ind_{0>\xi}\big|.$$
Indeed,
\begin{equation*}
\begin{split}
\stred\sup_{0\leq t\leq T}&\int_{\mt^N}|u^\varepsilon(t,x)|^p\,\dif x\leq \stred \sup_{0\leq t\leq T}\int_{\mt^N}\bigg(\int_\mr\big|\mathcal{S}(t,0)\ind_{u^\varepsilon_0>\xi}-\ind_{0>\xi}\big|\,\dif \xi\bigg)^p\dif x\\
&\leq C\,\stred \sup_{0\leq t\leq T}\big\|\mathcal{S}(t,0)\chi_{u^\varepsilon_0}(1+|\xi|)^p\big\|_{L^1_{x,\xi}}\\
&\qquad+ C\,\stred \sup_{0\leq t\leq T}\big\|\big(\mathcal{S}(t,0)\ind_{0>\xi}-\ind_{0>\xi}\big)(1+|\xi|)^p\big\|_{L^1_{x,\xi}}\\
&\leq C\Big(1+\|u_0^\varepsilon\|_{L^{r_p}_{\omega,x}}^{r_p}\Big).
\end{split}
\end{equation*}
The above exponents $r_p$ are given by Proposition \ref{chi} and the proof is complete.
\end{proof}
\end{prop}

\begin{cor}\label{ff}
For any $n\in[0,\infty)$ it holds
\begin{equation*}
\sup_{0\leq t\leq T}\stred\,\big\|h^\varepsilon(t)(1+|\xi|)^n\big\|_{L^1_{x,\xi}}\leq C.
\end{equation*}

\begin{proof}
It follows from \eqref{sol}, Proposition \ref{chi}, Corollary \ref{indik} and Proposition \ref{densities} that
\begin{equation*}
\begin{split}
\sup_{0\leq t\leq T}&\stred\,\big\|h^\varepsilon(t)(1+|\xi|)^n\big\|_{L^1_{x,\xi}}\leq \sup_{0\leq s\leq t\leq T}\stred\,\big\|\mathcal{S}(t,s)\chi_{u^\varepsilon(s)}(1+|\xi|)^n\big\|_{L^1_{x,\xi}}\\
&\qquad+\sup_{0\leq s \leq t\leq T}\stred\,\big\|\big(\ind_{0>\xi}-\mathcal{S}(s,0)\ind_{0>\xi}\big)(1+|\xi|)^n\big\|_{L^1_{x,\xi}}\\
&\leq C\Big(1+\sup_{0\leq s\leq T}\|u^\varepsilon(s)\|_{L^r_{\omega,x}}^r\Big)\leq C.
\end{split}
\end{equation*}
\end{proof}
\end{cor}

As a consequence, the assumptions of \cite[Theorem 5]{debus} are satisfied for $\nu^\varepsilon_{t,x}=\delta_{u^\varepsilon(t,x)=\xi}$ and hence there exists a kinetic measure $\nu_{t,x}$ vanishing at infinity such that $\nu^\varepsilon\rightarrow \nu$ in the sense given by this theorem. We deduce from \eqref{lll} that $\partial_\xi F=-\nu$ hence $F$ is a kinetic function in the sense of \cite[Definition 4]{debus}.

Remark, that it follows now from \eqref{meas} that the function $m^\varepsilon(t)$ satisfies
$$\sup_{0\leq t\leq T}\stred\,\big\|m^\varepsilon(t)(1+|\xi|)^n\big\|_{L^1_{x,\xi}}\leq C(\varepsilon),$$
for any $\varepsilon$ fixed. Nevertheless, we do not know yet if this fact holds true also uniformly in $\varepsilon$. Towards this end, we will study the weak formulation for $h^\varepsilon$ and employ a suitable test function.

\begin{prop}\label{kin}
For any $p\in[0,\infty)$ it holds
\begin{equation}\label{kin1}
\stred\int_{[0,T]\times\mt^N\times\mr}|\xi|^{2p}\,\dif m^\varepsilon(t,x,\xi)\leq C.
\end{equation}

\begin{proof}
Let $p\in[1/2,\infty)$. Regarding \eqref{formul1}, we need to test by $\varphi(\xi)=\frac{\xi^{2p+1}}{2p+1}$. Due to the behavior of $m^\varepsilon$ and $h^\varepsilon$ for large $\xi$ we can consider test functions which are not compactly supported in $\xi$, however, in this case the stochastic integral is not necessarily a martingale. Therefore we will first employ the truncation $\varphi^\delta(\xi)=\varphi(\xi)k_\delta(\xi)$ and then pass to the limit. We have
\begin{equation*}
\begin{split}
0\leq\stred\,\int_0^T\big\langle m^\varepsilon(t),\partial_\xi\varphi^\delta\big\rangle\,\dif t&=\stred\,\big\langle h_0^\varepsilon,\varphi^\delta\big\rangle-\stred\,\big\langle h^{\varepsilon}(T),\varphi^\delta\big\rangle\\
&\quad-\frac{1}{2}\,\stred\,\int_0^T\big\langle G^2\partial_\xi h^\varepsilon(t),\partial_\xi\varphi^\delta\big\rangle\,\dif t.
\end{split}
\end{equation*}
%Note, that due to the property $0\leq \sgn(\xi) f^\varepsilon(\xi)\leq 1$ {\color{red} toto uz mi nebude platit} the second term on the right hand side is nonpositive while
The first and the second term on the right hand side can be estimated by Corollary \ref{ff}
\begin{equation*}
\begin{split}
\stred\,\big\langle h_0^\varepsilon,\varphi^\delta\big\rangle-\stred\,\big\langle h^{\varepsilon}(T),\varphi^\delta\big\rangle
%\leq C\,\sup_{0\leq t\leq T}\stred\,\big\|h^\varepsilon(t)\,|\xi|^{2p+1}\big\|_{L^1_{x,\xi}}
\leq C,
%\stred\,\big\langle h_0^\varepsilon,\varphi^\delta\big\rangle&\leq C\,\stred\int_{\mt^N}\int_\mr\chi_{u_0^\varepsilon}\, \xi^{2p+1}\,\dif \xi\,\dif x\leq C\,\stred\|u_0^\varepsilon\|_{L^{2p+2}_x}^{2p+2}\leq C
\end{split}
\end{equation*}
while for the remaining term we first employ the growth properties of $G^2$ and $\partial_\xi G^2$ to obtain
\begin{equation*}
\begin{split}
\stred\,\int_0^T&\big\langle G^2\partial_\xi h^\varepsilon(t),\partial_\xi\varphi^\delta\big\rangle\,\dif t\\
&\leq C\,\stred\int_0^T\big\langle|h^\varepsilon(t)|,(1+|\xi|)\partial_\xi\varphi^\delta+(1+|\xi|^2)\partial_\xi^2\varphi^\delta\big\rangle\,\dif t\\
&\leq C\,\stred\int_0^T\big\langle|h^\varepsilon(t)|,(1+|\xi|)^{2p+3}\big\rangle\,\dif t\leq C.
\end{split}
\end{equation*}
The constant $C$ is independent of $\delta$ thus the claim follows.

If $p=0$ a suitable modification in the above estimation leads to the proof in this case whereas the case of $p\in(0,1/2)$ follows from \eqref{kin1} for $p=0$ and $p=1/2$ due to the fact that $|\xi|^{2p}\leq 1+|\xi|$.
%Therefore,
%$$\stred\int_{[0,T]\times\mt^N\times\mr}|\xi|^{2p}\,\dif m^\varepsilon(t,x,\xi)\leq C,\qquad\forall p\in[0,\infty).$$
\end{proof}
\end{prop}

Setting $p=0$ in \eqref{kin1} we regard $m^\varepsilon$ as random variables with values in $\mathcal{M}_b([0,T]\times\mt^N\times\mr)$, the space of bounded Borel measures on $[0,T]\times\mt^N\times\mr$ whose norm is given by the total variation of measures. We deduce that the set of laws $\{\prst\circ [m^\varepsilon]^{-1};\,\varepsilon\in(0,1)\}$ is tight and therefore any sequence has a weakly convergent subsequence due to the Prokhorov theorem. Consequently, the law of $m$ is supported in $\mathcal{M}_b([0,T]\times\mt^N\times\mr)$. Besides, $m$ is nonnegative as it holds true for all $m^\varepsilon$. Moreover, since $C_0([0,T]\times\mt^N\times\mr)$, the space of continuous functions vanishing at infinity equipped with the supremum norm, is the predual of $\mathcal{M}_b([0,T]\times\mt^N\times\mr)$ and $C_c^\infty([0,T]\times\mt^N\times\mr)$ is dense in $C_0([0,T]\times\mt^N\times\mr)$ it can be seen that \eqref{ddr} holds true for any $\varphi\in C_0([0,T]\times\mt^N\times\mr)$.
Now, it is left to verify the three points of the definition of a kinetic measure \cite[Definition 1]{debus}. The second requirement giving the behavior for large $\xi$ follows from the above uniform estimate \eqref{kin1}. Indeed, let $(k_\delta)$ be a truncation on $\mr$, e.g. the set of functions defined in the proof of Proposition \ref{oper}, then
\begin{equation*}
\begin{split}
\stred\int_{[0,T]\times\mt^N\times\mr}&|\xi|^{2p}\,\dif m(t,x,\xi)\leq\liminf_{\delta\rightarrow0}\stred\int_{[0,T]\times\mt^N\times\mr}|\xi|^{2p}k_\delta(\xi)\,\dif m(t,x,\xi)\\
&=\liminf_{\delta\rightarrow0}\lim_{\varepsilon\rightarrow 0}\stred\int_{[0,T]\times\mt^N\times\mr}|\xi|^{2p}k_\delta(\xi)\,\dif m^\varepsilon(t,x,\xi)\leq C.
\end{split}
\end{equation*}
As a consequence, $m$ vanishes for large $\xi$. The first point of \cite[Definition 1]{debus} is straightforward for $\phi\in C_0([0,T]\times\mt^N\times\mr)$ as a pointwise limit of measurable functions is measurable. The case of $\phi\in C_b([0,T]\times\mt^N\times\mr)$ now follows by employing the truncation $(k_\delta)$ together with the dominated convergence theorem as $\delta\rightarrow 0$ and the behavior of $m$ at for large $\xi$. In order to show predictability of the process
$$t\longmapsto\int_{[0,t]\times\mt^N\times\mr}\phi(x,\xi)\,\dif m(s,x,\xi)$$
in the case of $\phi\in C_0(\mt^N\times\mr)$ let us remark that due to \eqref{formul1} it is the pointwise limit (in $\omega$ and $t$) of predictable processes
$$t\longmapsto \int_{[0,t]\times\mt^N\times\mr}\phi(x,\xi)\,\dif m^\varepsilon(s,x,\xi)$$
and hence is also measurable with respect to the predictable $\sigma$-algebra.
%These processes are predictable due to the definition of measures $m^\varepsilon$ in \eqref{meas}. Let $\alpha\in L^2(\Omega),\,\gamma\in L^2(0,T)$, then, by the Fubini theorem,
%$$\stred\bigg(\alpha\int_0^T\gamma(t) x^\varepsilon(t)\,\dif t\bigg)=\stred\bigg(\alpha\int_{[0,T]\times\mt^N\times\mr}\phi(x,\xi)\Gamma(s)\,\dif m^\varepsilon(s,x,\xi)\bigg)$$
%where $\Gamma(s)=\int_s^T\gamma(t)\,\dif t.$ Hence, since $\Gamma$ is continuous, we obtain by the weak convergence of $m^\varepsilon$ to $m$
%$$\stred\bigg(\alpha\int_0^T\gamma(t) x^\varepsilon(t)\,\dif t\bigg)\longrightarrow\stred\bigg(\alpha\int_0^T\gamma(t) x(t)\,\dif t\bigg).$$
%Consequently, $x^\varepsilon$ converges to $x$ weakly in $ L^2(\Omega\times[0,T])$ and, in particular, since the space of predictable processes is weakly closed, it follows that also $x$ is predictable.
The case of $\phi\in C_b(\mt^N\times\mr)$ can be verified by using truncations as above. Therefore, we have proved that $m$ is a kinetic measure.

Finally, we deduce that $F$ satisfies the generalized kinetic formulation \eqref{general} and thus is a generalized kinetic solution to \eqref{conser}. Since any generalized kinetic solution is actually a kinetic one, due to the reduction theorem \cite[Theorem 11]{debus}, it follows that $F=\ind_{u>\xi}$ and $\nu=\delta_u$, where $u\in L^p(\Omega\times[0,T]\times\mt^N)$ is the unique kinetic solution to \eqref{conser}. Therefore, it only remains to verify the strong convergence of $f^\varepsilon$ and $u^\varepsilon$ to $\chi_u$ and $u$, respectively.

According to \eqref{weakstar}, we deduce for $f^\varepsilon=F^\varepsilon-\ind_{0>\xi}$ that
\begin{equation*}\label{fugu}
\begin{split}
f^\varepsilon\overset{w^*}\longrightarrow\chi_u\qquad\text{in}\qquad L^\infty(\Omega\times[0,T]\times\mt^N\times\mr),\\
\end{split}
\end{equation*}
and by \eqref{klad} it holds
\begin{equation*}
\begin{split}
\chi_{u^\varepsilon}\longrightarrow\chi_u\qquad\text{in}\qquad\mathcal{D}'((0,T)\times\mt^N\times\mr),\,\prst\text{-a.s.}.
\end{split}
\end{equation*}
Besides, $\{\chi_{u^\varepsilon};\,\varepsilon\in(0,1)\}$ is bounded in $L^\infty(\Omega\times[0,T]\times\mt^N\times\mr)$ hence (up to subsequences) it converges weak* in this space and since $C_c^\infty((0,T)\times\mt^N\times\mr)$ is separable and dense in $L^1([0,T]\times\mt^N\times\mr)$, it follows that $\chi_u$ is the limit, i.e.
$$\chi_{u^\varepsilon}\overset{w^*}{\longrightarrow}\chi_u\qquad\text{in}\qquad L^\infty(\Omega\times[0,T]\times\mt^N\times\mr).$$
Furthermore, according to Proposition \ref{densities}, it holds for any $n\in[0,\infty)$
\begin{equation}\label{att1}
\sup_{0\leq t\leq T}\stred\int_{\mt^N}\int_{\mr}\big(|\chi_{u^\varepsilon(t)}|+|\chi_{u(t)}|\big)(1+|\xi|)^n\,\dif \xi\,\dif x\leq C,
\end{equation}
hence we can relax the conditions on test functions and obtain the strong convergence $\chi_{u^\varepsilon}\rightarrow \chi_u$ in $L^2(\Omega\times[0,T]\times\mt^N\times\mr)$. Indeed,
\begin{equation}\label{att}
\begin{split}
\stred\int_0^T\int_{\mt^N}&\int_\mr|\chi_{u^\varepsilon}-\chi_u|^2\,\dif\xi\,\dif x\,\dif t\\
&=\stred\int_0^T\int_{\mt^N}\int_\mr|\chi_{u^\varepsilon}|-2\chi_{u^\varepsilon}\chi_{u}+|\chi_{u}|\,\dif\xi\,\dif x\,\dif t\longrightarrow0
%\leq \stred\int_0^T\int_{\mt^N}\int_{-K}^K|\chi_{u^\varepsilon}-\chi_u|\,\dif\xi\,\dif x\,\dif t\\
%&\hspace{1.7cm}+\frac{1}{K}\,\stred\int_0^T\int_{\mt^N}\int_\mr|\chi_{u^\varepsilon}-\chi_u|\,|\xi|\,\dif\xi\,\dif x\,\dif t.
\end{split}
\end{equation}
since for the first term on the right hand side we have
\begin{equation*}
\begin{split}
\stred\int_0^T\int_{\mt^N}\int_\mr|\chi_{u^\varepsilon}|\,\dif\xi\,\dif x\,\dif t&=\stred\int_0^T\int_{\mt^N}\int_\mr\big(\chi_{u^\varepsilon}\ind_{\xi>0}-\chi_{u^\varepsilon}\ind_{\xi<0}\big)\,\dif\xi\,\dif x\,\dif t
\end{split}
\end{equation*}
where $\ind_{\xi>0},\ind_{\xi<0}$ can be taken as test functions due to \eqref{att1} and for the second term on the right hand side we consider $\chi_u$ as a test function.
%{\color{red} to neni spravne, nemam konvergenci $L^1_{loc}$ musi se jit asi pres $L^2$, mozna to preci jen projde, protoze
%$$\bigg|\int_\mr \chi_\alpha-\chi_\beta\bigg|=\int_\mr|\chi_\alpha-\chi_\beta|$$}
%The first term on the right hand side vanishes as $\varepsilon\rightarrow 0$ for any $K>0$ whereas the convergence of the second term to zero holds for $K\rightarrow\infty$ uniformly in $\varepsilon$.
As $|\chi_\alpha-\chi_\beta|^p=|\chi_\alpha-\chi_\beta|$ we conclude also the strong convergence in all $L^p(\Omega\times[0,T]\times\mt^N\times\mr)$, $p\in[1,\infty)$.

Moreover, a similar approach can be used to prove the convergence of $f^\varepsilon$. Indeed, the same calculation as in \eqref{att} gives
$$f^\varepsilon\longrightarrow \chi_u\qquad\text{in}\qquad L^2(\Omega\times[0,T]\times\mt^N\times\mr)$$
and using the uniform bound of $\{f^\varepsilon;\,\varepsilon\in(0,1)\}$ in $L^\infty(\Omega\times[0,T]\times\mt^N\times\mr)$ we deduce the convergence in $L^p(\Omega\times[0,T]\times\mt^N\times\mr)$ for all $p\in[1,\infty)$.
%according to the uniform estimate \eqref{ffest}, we can test by constants (or for example by $\varphi=\ind_{0<\xi},\,\ind_{0>\xi}$) in the weak*-convergence \eqref{fugu}. Hence, since $|f^\varepsilon|^2\leq |f^\varepsilon|$ and $0\leq \sgn(\xi) f^\varepsilon(\xi)\leq 1$ holds true we have
%\begin{equation*}
%\begin{split}
%\stred\int_0^T\int_{\mt^N}\int_{\mr}|f^\varepsilon-\chi_u|^2\,\dif\xi\,\dif x\,\dif t&\leq\stred\int_0^T\int_{\mt^N}\int_{\mr} |f^\varepsilon|-2f^\varepsilon\chi_u+\chi_u\,\dif\xi\,\dif x\,\dif t
%\end{split}
%\end{equation*}
%where the right hand side tends to zero.
%Therefore, 
%$$f^\varepsilon\longrightarrow \chi_u\qquad\text{in}\qquad L^2(\Omega\times[0,T]\times\mt^N\times\mr)$$
%and the convergence in $L^p_{loc}(\Omega\times[0,T]\times\mt^N\times\mr)$, $p\in[1,\infty),$ now follows easily from the uniform bound of $\{f^\varepsilon;\,\varepsilon\in(0,1)\}$ in $L^\infty(\Omega\times[0,T]\times\mt^N\times\mr)$. Next, we observe that
%\begin{equation*}
%\begin{split}
%\stred\int_0^T\int_{\mt^N}\int_\mr|f^\varepsilon-\chi_u|&\,\dif\xi\,\dif x\,\dif t\leq\stred\int_0^T\int_{\mt^N}\int_{-K}^K|f^\varepsilon-\chi_u|\,\dif\xi\,\dif x\,\dif t\\
%&+\frac{1}{K}\,\stred\int_0^T\int_{\mt^N}\int_\mr|f^\varepsilon-\chi_u|\,|\xi|\,\dif\xi\,\dif x\,\dif t.
%\end{split}
%\end{equation*}
%The first term on the right hand side vanishes as $\varepsilon\rightarrow0$ for any $K>0$ whereas the convergence of the second term to zero holds for $K\rightarrow\infty$ uniformly in $\varepsilon$. Therefore,

Eventually, by the properties of the equilibrium function we have
$$u^\varepsilon\longrightarrow u\qquad\text{in}\qquad L^1(\Omega\times[0,T]\times\mt^N).$$
On the other hand, it follows from Proposition \ref{densities} that the set $\{u^\varepsilon;\,\varepsilon\in(0,1)\}$ is bounded in $L^p(\Omega\times[0,T]\times\mt^N),$ for all $p\in[1,\infty),$ hence by application of the H\"older inequality, we get also the strong convergence
$$u^\varepsilon\longrightarrow u\qquad\text{in}\qquad L^p(\Omega\times[0,T]\times\mt^N)\qquad\forall p\in[1,\infty).$$
Therefore, the proof of convergence in the stochastic BGK model is complete.
\end{proof}

\subsection*{Acknowledgment}
The author is greatly indebted to Arnaud Debussche and Jan Seidler for many stimulating discussions. This research was supported in part by the University Center for Mathematical Modelling, Applied Analysis and Computational Mathematics (Math MAC), Czech Republic, and Inria, Centre de Recherche, Rennes -- Bretagne Atlantique, France.
%Thanks go also to Josef M\'alek and University Center for Mathematical Modelling, Applied Analysis and Computational Mathematics (Math MAC) for their support.

\end{document}